\documentclass[10pt]{elsarticle}
\usepackage{amssymb,amsmath,amsfonts,mathrsfs,amsthm,mathtools}
\usepackage{graphicx}
\usepackage{bm}
\usepackage[margin=1in,papersize={8.5in,11in}]{geometry}
\usepackage[utf8]{inputenc}
\usepackage[english]{babel}
\usepackage{xcolor}

\newtheorem{prop}{Proposition}
\makeatletter

\def\H{\mathcal{H}}
\def\D{\mathcal{D}}
\def\L{\mathcal{L}} 
\def\M{\mathcal{M}} 
\def\P{\mathcal{P}}
\def\Q{\mathcal{Q}}

\def\I{\mathcal{I}}

\journal{Arxiv}

\begin{document}

\begin{frontmatter}
\title{Generalized Langevin equations for systems with local interactions}
\author[ucsc]{Yuanran Zhu}
\author[ucsc]{Daniele Venturi\corref{correspondingAuthor}}
\address[ucsc]{Department of Applied Mathematics, University of California Santa Cruz\\ Santa Cruz (CA) 95064}
\cortext[correspondingAuthor]{Corresponding author}
\ead{venturi@ucsc.edu}
 
\begin{abstract}
We present a new method to approximate 
the Mori-Zwanzig (MZ) memory integral 
in generalized Langevin equations (GLEs) describing 
the evolution of smooth observables in 
high-dimensional nonlinear systems with 
local interactions. 
Building upon the Faber operator series 
we recently developed for the orthogonal dynamics 
propagator, and an exact combinatorial algorithm 
that allows us to compute memory kernels from first 
principles, we demonstrate that the proposed 
method is effective in computing auto-correlation 
functions, intermediate scattering functions and other 
important statistical properties of the observable.
{  We also develop a new stochastic process 
representation of the MZ fluctuation term for systems 
in statistical equilibrium}. Numerical 
applications are presented for the Fermi-Pasta-Ulam model, and 
for random wave propagation in homogeneous media.
\end{abstract}
\end{frontmatter}

\section{Introduction}

The Mori-Zwanzig (MZ) formulation is a technique 
developed in statistical mechanics to formally 
integrate out phase variables in nonlinear 
dynamical systems by means of a projection operator. 
One of the main features of such formulation is that 
it allows us to systematically derive formally exact generalized 
Lagevin equations (GLEs) for quantities of interest 
(macroscopic observables), based on microscopic 
equations of motion. Such GLEs can be 
found in a variety of applications, including particle 
dynamics  \cite{snook2006langevin,Yoshimoto2013,
Li2015,Li2017,lei2016data}, partial differential 
equations (PDEs) \cite{Chorin,Stinis1,chertock2008modified,
ChoPRS2014,VenturiBook}, fluid dynamics 
\cite{parish2017dynamic,parish2017non},  and solid-state 
physics \cite{Woo2017,Li2010,mendl2015current}.
As an example, consider the Brownian motion 
of a colloidal particle subject to collision interactions 
with a large number of fluid particles. By using the MZ 
formulation it is possible to derive a low-dimensional 
system of equations characterizing the dynamics 
(position and momentum) of the colloidal particle 
alone \cite{van1986brownian,lei2016data}.

Computing the solution to the MZ equation 
is a challenging task. One of the main difficulties is 
the approximation of the memory integral 
(convolution term), and the fluctuation term (noise term). 
These terms encode the interaction between the 
so-called orthogonal dynamics and the 
dynamics of the quantity of interest. The orthogonal dynamics is 
essentially a high-dimensional nonlinear flow 
that satisfies a hard-to-solve integro-differential equation. 
Such flow has, in general, the same order of magnitude and 
dynamical properties as the quantity of interest, i.e., 
there is no general scale separation between the so-called 
resolved and the unresolved variables \cite{Chorin1,Stinis}. 
As a consequence, approximating the MZ memory integral and the 
fluctuation term  in these cases is often a daunting task, 
because of the strong coupling between the orthogonal 
dynamics and the macroscopic observables.  
Over the years, many techniques have been 
proposed to address this problem. These techniques 
can be grouped in two categories: i) data-driven 
methods;  ii) methods based on first-principles. 
Data-driven methods aim at recovering the MZ memory 
integral/fluctuation term based on data, usually in the form 
of sample trajectories of the full system. 
Typical examples are the NARMAX 
technique developed by Lu {\em et al.} \cite{lu2017data}, 
the rational function approximation recently 
proposed by Lei {\em et al.} \cite{lei2016data} 
(see also \cite{chu2017mori}), and the conditional 
expectation technique developed by Brennan 
and Venturi \cite{Brennan2018}.
On the other hand, methods based on first principles 
aim at approximating the MZ memory integral 
and fluctuation term based on the structure of 
the nonlinear system (microscopic equations of motion), 
without using any simulation data. 
The first effective method developed within this class is 
the continued fraction expansion of Mori \cite{mori1965continued},  
which can be conveniently formulated in terms of 
recurrence relations \cite{lee1982solutions,florencio1985exact}.
Other methods based on first-principles 
include perturbation methods 
\cite{watts1977perturbation,venturi2014convolutionless},  
mode coupling techniques, \cite{reichman2005mode,gotze1999recent}, 
optimal prediction methods \cite{Chorin,Chorin1,Stinis,chertock2008modified}, and 
various series expansion
\cite{stinis2015renormalized,parish2017non,parish2017dynamic,
zhu2018estimation}. 
First-principle calculation methods can effectively 
capture non-Markovian memory 
effects, e.g., in coarse-grained particle simulations 
\cite{Yoshimoto2013,hijon2010mori}. 
However, they are often quite involved 
and they do not generalize well to systems with 
no scale separation \cite{givon2004extracting}.  
At the same time, data-driven methods 
can yield accurate results, but they often require a 
large number of sample trajectories to faithfully 
capture memory effects \cite{Brennan2018,chu2017mori,lei2016data,Li2015,Li2017} 

In this paper, we present a new method to compute the 
MZ memory integral and the flucuation term 
from first principles in nonlinear systems with 
local polynomial interactions. 
To this end, we build upon the Faber operator series expansion 
we recently developed in \cite{zhu2018faber}, and a new 
combinatorial algorithm that allows us to compute 
the MZ memory kernel by using only 
the structure of the microscopic equations of motion.  
We also develop a new data-driven stochastic process representation 
method based on the MZ memory kernel and Karhunen-Lo\`eve 
(KL) series expansions, which allows us to build simple 
models of the MZ fluctuation term in systems with invariant 
measures, e.g., Hamiltonian systems or more general systems 
\cite{bouchet2010invariant,flandoli1994dissipativity}.

This paper is organized as follows. 
In Section \ref{sec:review} we briefly review the MZ 
formulation for nonlinear dynamical systems evolving 
from random initial states. In Section \ref{sec:MGLE} 
we specialize such formulation by introducing Mori's 
projection operators. In Section \ref{sec:MZseries} 
and Section \ref{sec:recursive} we develop series 
expansion of the MZ memory kernel based on the Faber 
operator series we recently proposed in 
\cite{zhu2018faber}. Section \ref{sec:calculation_of_gamma}
is devoted to the description of an exact combinatorial algorithm to 
compute the recurrence coefficients of the MZ 
memory kernel expansion. In Section \ref{sec:Model}, we 
develop a new { stochastic process representation 
method to compute the MZ fluctuation term 
for systems in statistical equilibrium}. In Section
\ref{sec:application} we demonstrate the accuracy 
of the MZ memory calculation and the reduced-order 
stochastic modeling technique in applications to 
nonlinear random wave propagation 
described by Hamiltonian partial differential equations. 
The main findings of the paper are summarized 
in Section \ref{sec:conclusion}. 
We also include a brief Appendix where prove convergence 
of KL expansions in representing auto-correlation functions 
of polynomial observables.

\section{The Mori-Zwanzig formulation}
\label{sec:review}
Consider the following nonlinear dynamical system 
evolving on a smooth manifold $\mathcal{M}\subseteq \mathbb{R}^N$
\begin{equation}
 \label{eqn:nonautonODE}
 \frac{d{\bm x}}{dt} = \bm {F}(\bm x),\qquad 
{\bm x}(0) = \bm x_0,
\end{equation}
where $\bm x_0\in \M$ is a random initial state with probability density 
function $\rho_0(\bm x)$.
The dynamics of any vector-valued phase space function 
\begin{align}
\bm u\colon \mathcal{M} &\to \mathbb{R}^M \nonumber\\
\bm x & \mapsto \bm u(\bm x)
\label{observable}
 \end{align}
can be expressed in terms of a semi-group of linear 
operators acting on $\bm u(\bm x_0)$, i.e., 
\begin{equation}
\bm u(\bm x(t,\bm x_0))=e^{t\L(\bm x_0)} \bm u(\bm x_0),\qquad 
\L(\bm x_0) = \sum_{k=1}^N F_k(\bm x_0)\frac{\partial }{\partial x_{0k}}.
\label{Koopman}
\end{equation}
In this equation, $\bm x(t,\bm x_0)$ represents the 
flow \cite{Wiggins,Hirsch} generated by the 
system \eqref{eqn:nonautonODE}, while 
$e^{t\L}$ is the composition (Koopman) operator
of the system \cite{Koopman1931,dominy2017duality}.
We are interested in deriving the exact evolution equation 
for the phase space function $\bm u(t)=\bm u(\bm x(t,\bm x_0))$. 
To this end, we employ the Mori-Zwanzig formulation 
\cite{zwanzig1973nonlinear,Chorin,zhu2018estimation}. 
The first step is to introduce an orthogonal projection 
operator $\P$, and the complementary projection 
$\Q=\I-\P$, where $\I$ is the identity operator.
The mathematical properties of such projections 
are discussed in detail 
in \cite{zhu2018estimation,dominy2017duality}.
By differentiating the well-known Dyson's identity 
\begin{align}
e^{t\L}=e^{t\Q\L}+
\int_0^t e^{s\L}\P\L e^{(t-s)\Q\L}ds
\label{Dyson}
\end{align}
with respect to time, we obtain the following evolution 
equation for the Koopman operator $e^{t\L}$
\begin{align}
\frac{d}{dt} e^{t\L} = e^{t\L}\P\L + 
e^{t\Q\L}\Q\L+ \int_0^t e^{s\L}\P\L e^{(t-s)\Q\L}\Q\L \,ds.
\label{MZKoop}
\end{align}
Applying this equation to any phase space 
function $\bm u(0)=\bm u(\bm x_0)$ yields the
Mori-Zwanzig (MZ) equation 
\begin{align}
\frac{\partial}{\partial t}e^{t\mathcal{L}}\bm u(0)
&=e^{t\mathcal{L}}\mathcal{PL}\bm u(0)
+e^{t\mathcal{QL}}\mathcal{QL}\bm u(0)+\int_0^te^{s\mathcal{L}}\mathcal{PL}
e^{(t-s)\mathcal{QL}}\mathcal{QL}\bm u(0)ds.
\label{MZstatespace}
\end{align}
The three terms at the right hand side are  
called, respectively, streaming term, fluctuation 
(or noise) term, and memory term. It is often more 
convenient (and tractable) to compute the evolution 
of the observable $\bm u(t)$  within a closed 
linear space, e.g., the image of the projection operator 
$\P$. To this end, we apply such projection to both sides of equation 
\eqref{MZstatespace}. This yields the following exact evolution 
equation\footnote{Note that the projected fluctuation term 
$\P e^{t\mathcal{QL}}\mathcal{QL}\bm u(0)$ is identically 
zero since $\P\Q=0$.}
\begin{align}
\label{reduced order equation}
\frac{\partial}{\partial t}\mathcal{P} e^{t\mathcal{L}}\bm u(0)
=\mathcal{P}e^{t\mathcal{L}}\mathcal{PL}\bm u(0)
+\int_0^t\mathcal{P}e^{s\mathcal{L}}\mathcal{PL}
e^{(t-s)\mathcal{QL}}\mathcal{QL}\bm u(0)ds.
\end{align}
Depending on the choice of the projection operator, the MZ 
equation \eqref{reduced order equation} can yield evolution 
equations for different quantities. For example, if we use 
Chorin's projection 
\cite{Chorin,chorin2002optimal,zhu2018estimation,VenturiBook}, 
then \eqref{reduced order equation} is an evolution 
equation for the conditional mean of $\bm u(t)$. Similarly, 
if we use Mori's projection \cite{zhu2018faber,snook2006langevin}, 
then \eqref{reduced order equation} is an evolution 
equation for the temporal auto-correlation function 
of $\bm u(t)$.

\subsection{Mori's projection operator}
\label{sec:MGLE}
Suppose that the phase space function \eqref{observable} 
belongs to the weighted Hilbert space $H=L^2(\M,\rho)$, 
where $\rho$ is a positive weight function in $\M$. 
For instance, $\rho$ can be the 
probability density function of the random initial state 
$\bm x_0$ (i.e., $\rho_0$, see Eq. 
\eqref{eqn:nonautonODE}), 
or the equilibrium distribution of 
the system $\rho_{eq}$ (assuming it exists).  
Let
\begin{equation}
\langle f,g\rangle_{\rho}=\int_{\M} 
f(\bm x)g(\bm x)\rho(\bm x)d\bm x 
\qquad f,g\in H
\label{ip}
\end{equation}
be the inner product in $H$. 
For any closed linear 
subspace $V\subset H$ 
the Mori projection operator $\P$ is defined 
to be the orthogonal projection onto $V$, 
relative to the inner product \eqref{ip}. 
If $V$ is finite-dimensional with dimension $M$, 
then $\P$ can be effectively constructed if we are given 
$M$ linearly independent functions 
$u_i(0)=u_i(\bm x)\in V$ ($i=1,...,M$). 
Clearly, if $\{u_1(0),\dots, u_M(0)\}$ 
are linearly independent then 
$V=\text{span}\{u_1(0),\dots,u_M(0)\}$.
To construct Mori's projection, we first 
compute the positive-definite Gram matrix 
$G_{ij}=\langle u_i(0),u_j(0)\rangle_{\rho}$, i.e., 
\begin{equation}
G_{ij}= \int_{\M} u_i(\bm x)u_j(\bm x)\rho(\bm x)d\bm x.
\end{equation}
With $G_{ij}$ available, we define
\begin{align}
\label{Mori_P}
\P f=\sum_{i,j=1}^M G^{-1}_{ij}
\langle u_i(0),f\rangle_{\rho}u_j(0),
\qquad f\in H.
\end{align}
In classical statistical dynamics of Hamiltonian systems, 
a common choice for the density $\rho$ is the Boltzmann-Gibbs 
distribution
\begin{equation}
\rho_{eq}(\bm x)=\frac{1}{Z}e^{-\beta\H(\bm x)},
\label{rhoeq}
\end{equation}  
where $\H(\bm x)=\H(\bm q,\bm p)$ is the Hamiltonian 
of the system, $\bm x=(\bm q,\bm p)$ are generalized coordinates/momenta, and $Z$ is the partition function. For other 
systems, $\rho$ can be, e.g., the probability density 
function of the random initial state (see Eq. \eqref{eqn:nonautonODE}).  
Next, suppose that each observable 
$u_i(\bm x)$ ($i=1,\dots, M$) belongs to the linear 
space $\P H \cap\D(\L)$, where $\P H= V$ 
and $\D(\L)$ denotes the domain of the 
Liouville operator $\L$ defined in \eqref{Koopman}.
The MZ equation \eqref{MZstatespace}, with 
$\P$ defined in \eqref{Mori_P}, reduces to 
\begin{align}\label{gle_full}
\frac{d \bm{u}(t) }{dt}= \bm \Omega\bm{u}(t)+
\int_{0}^{t}\bm K(t-s)\bm{u}(s)ds+\bm f(t),
\end{align}
where\footnote{Note that the $i$th 
component of the  system \eqref{gle_full} 
can be explicitly written as 
\begin{equation}
\frac{du_i (t)}{dt} = \sum_{j=1}^M \Omega_{ij} u_j(t)+
\sum_{j=1}^M \int_{0}^{t}K_{ij}(t-s)u_j(s)ds+f_i(t).
\label{gle_full1}
\end{equation}
}
\begin{subequations}
\begin{align}
		G_{ij} & = \langle u_{i}(0), u_{j}(0)\rangle_{\rho}\quad 
		\text{(Gram matrix)},\label{gram}\\
		\Omega_{ij} &= \sum_{k=1}^MG^{-1}_{jk}
		\langle u_{k}(0), \L u_{i}(0)\rangle_{{\rho}}\quad 
		\text{(streaming matrix)},\label{streaming}\\
		K_{ij}(t) & =\sum_{k=1}^M G^{-1}_{jk}
		\langle u_{k}(0), \L e^{t\Q\L}\Q\L u_{i}(0)\rangle_{\rho}\quad 
		\text{(memory kernel)},\label{SFD}\\
		 \bm f(t)& =e^{t\Q\L}\Q\L \bm u(0) \quad 
		\text{(fluctuation term)}.\label{f}
	\end{align}
\end{subequations}
Equation \eqref{gle_full} is often referred to as 
generalized Langevin equation (GLE) in classical statistical 
physics and other disciplines \cite{snook2006langevin}. 
By applying Mori's projection to \eqref{gle_full} we obtain 
the following linear (and closed) evolution equation for the 
projected phase space function 
\begin{align}
\label{gle}
\frac{d}{dt}\P\bm{u}(t) = \bm \Omega\P \bm{u}(t) +
\int_{0}^{t}\bm K(t-s) \P \bm{u}(s)\,ds.
\end{align}
Acting with the inner product 
$\displaystyle \langle u_j(0),\cdot\rangle_{\rho}$ 
on both sides of equation \eqref{gle}, yields the following 
exact equation for the temporal auto-correlation matrix 
$C_{ij}(t) =\langle u_j(0), u_i(t)\rangle_{\rho}$ 
\begin{align}
\label{gle_C}
\frac{d}{dt}C_{ij}(t) = \sum_{k=1}^M \Omega_{ik}C_{kj}(t) + 
\sum_{k=1}^M\int_{0}^{t} K_{ik}(t-s) C_{kj}(s)ds.
\end{align}
Suppose that the system \eqref{eqn:nonautonODE} 
is Hamiltonian, and that the random initial state $\bm x_0$ 
is distributed according to the Boltzmann-Gibbs distribution 
\eqref{rhoeq}, i.e., $\rho_0=\rho_{eq}$. 
In these assumptions, the Liouville operator $\L$
is skew-adjoint relative to the inner 
product \eqref{ip}, i.e., we have
\begin{equation}
\langle f,\L g \rangle_{eq} = - \langle \L f,g \rangle_{eq} \qquad f,g\in L^2(\M,\rho_{eq})\cap\D(\L).
\end{equation}
This allows us to simplify the expression of the 
memory kernel \eqref{SFD} as 
\begin{align}
K_{ij}(t)=&-\sum_{k=1}^M G^{-1}_{jk}\langle\Q\L u_{k}(0),  e^{t\Q\L}\Q\L u_{i}(0)\rangle_{eq},\nonumber \\
=&-\sum_{k=1}^MG^{-1}_{jk}\langle f_k(0),f_i(t) \rangle_{eq},
\label{2nd_FDT}
\end{align}
where $f_k(t)$ is the $k$-th component of the fluctuation term 
\eqref{f}. The identity \eqref{2nd_FDT} is known as Kubo's second 
fluctuation-dissipation theorem \cite{kubo1966fluctuation}.
We emphasize there are several advantages in 
using Mori's projection \eqref{Mori_P}
over other projection operators, e.g., 
Chorin's projection \cite{chorin2000optimal}. 
{  For example, both MZ equations \eqref{gle_full} 
and \eqref{gle} are linear and closed}, 
which allows us perform rigorous convergence 
analysis  \cite{zhu2018faber,zhu2018estimation}. 
Secondly, the streaming matrix \eqref{streaming} and the memory 
kernel \eqref{SFD} are exactly the same for both the projected 
and the unprojected equations ,i.e.,  \eqref{gle_full} and 
\eqref{gle}). Thirdly, we have that the second-fluctuation 
dissipation theorem \eqref{2nd_FDT} holds true, 
which allows us to express the MZ memory kernel 
in a relatively simple form in terms of averages 
of random forces.

\subsection{Series expansion of the MZ memory kernel}
\label{sec:MZseries}
To compute the solution of the Mori-Zwanzig equation \eqref{gle} 
we need to evaluate the memory kernel \eqref{SFD}. This is often 
a daunting task due to the presence of terms such 
as $e^{t\Q\L} u_i(0)$, i.e., terms involving operator 
exponentials. 
A straightforward method to evaluate such terms is to expand
the orthogonal dynamics propagator $e^{t\Q\L}$ 
in a truncated operator polynomial series as\footnote{
The series expansion \eqref{generalS} needs to be 
handled with care if $\L$ is 
an unbounded operator, e.g., the 
generator of the Koopman semigroup \eqref{Koopman} (see \cite{kato2013perturbation}, p. 481). 
In this case, $e^{t\L}$  should be properly defined as
\begin{equation}
e^{t\L}=\lim_{q\rightarrow \infty}\left(1-\frac{t\L}{q}\right)^{-q}.
\nonumber
\end{equation}
In fact, $\left(1-t\L /q\right)^{-1}$ is the resolvent of $\L$ 
(modulus a constant factor), which can be rigorously 
defined for both bounded and unbounded operators.
}
\begin{equation}
e^{t\Q\L} \simeq \sum_{q=0}^n g_q(t)\Phi_q\left(\Q\L\right).
\label{generalS}
\end{equation}
The temporal modes $g_q(t)$ can be explicitly computed for any 
specific choice of the polynomial set $\{\Phi_0,\dots,\Phi_n\}$. 
For example, if $\Phi_q$ ($q=0,\dots,n$) 
are Faber polynomials \cite{zhu2018faber,Novati2003},
then $g_q(t)$ are products of exponential functions and Bessel functions 
of the first kind. Similarly, if $\Phi_q\left(\Q\L\right)=(\Q\L)^q$, 
then $g_q(t)=t^q/q!$ (see Table \ref{tab:1}).
A substitution of \eqref{generalS} into \eqref{SFD} 
yields the following series expansion of the MZ memory kernel
\begin{align}
K_{ij}(t) =& \sum_{k=1}^M G^{-1}_{jk}
		\langle u_{k}(0), \L e^{t\Q\L}\Q\L u_{i}(0)\rangle_{\rho},\nonumber\\
		\simeq &  \sum_{q=0}^n g_q(t) \sum_{k=1}^M G^{-1}_{jk}
		\langle u_{k}(0), \L \Phi_q(\Q\L)\Q\L u_{i}(0)\rangle_{\rho},\nonumber\\
		=&\sum_{q=0}^n g_q(t)M_{ijq},
		\label{K_formula}
\end{align}
where 
\begin{equation}
M_{ijq}=\sum_{k=1}^M G^{-1}_{jk}
		\langle u_{k}(0), \L \Phi_q(\Q\L)\Q\L u_{i}(0)\rangle_{\rho}.
		\label{Mdef}
\end{equation}
 {Note that the coefficients $M_{ijq}$ are 
 generalized operator cumulants in the sense of Hegerfeldt and Kubo  \cite{Hegerfeldt,kubo1962generalized,snook2006langevin}. Such 
Coefficients can be computed in a data-driven setting \cite{lei2016data,berkowitz1981memory},
or based on first-principles as we describe 
in Section \ref{sec:recursive}.}
We also emphasize that, in general, series expansions 
of the orthogonal dynamics propagator \eqref{generalS} 
with different polynomial bases $\Phi_q$ can yield different 
approximation errors in the MZ memory kernel 
\eqref{K_formula} for the same polynomial 
order $n$ (see \cite{zhu2018faber}).

{ Regarding convergence of \eqref{generalS}, 
our recent analysis \cite{zhu2018faber,zhu2018estimation} 
demonstrates that it can be rigorously established for 
linear systems and arbitrary integration times 
(see \S 5 in \cite{zhu2018faber}).
If the system is nonlinear, then the series expansion 
of the memory kernel \eqref{K_formula} is granted 
to converge only on a time interval that depends on the 
system and on the observable $\bm u(t)$. In particular, 
Corollary 3.4.3 in \cite{zhu2018estimation} establishes
short-time convergence of the MZ memory 
approximation \eqref{K_formula} for a broad 
class of nonlinear systems of the form \eqref{eqn:nonautonODE}. 
This implies that such memory approximation might 
be accurate only for a short integration time that depends 
on the system and the observable.}

\begin{table}
\begin{minipage}{4.0cm}
\small
\vspace{0.6cm}MZ memory kernel \\
$K_{ij}(t-s) \simeq \displaystyle \sum_{q=0}^{n} g_q(t-s) M_{ijq}$\\
\end{minipage}
\hspace{0.4cm}
\begin{minipage}{3.0cm}
\label{table}
\centering\small
\begin{tabular}{l l l}
Type & Temporal modes $g_q(t)$ & Coefficients $M_{ijq}$\vspace{0.1cm} \\
\hline\\
MZ-Dyson    \hspace{0.5cm} & $\displaystyle \frac{t^q}{q!}$   & $\displaystyle \sum_{k=1}^M G^{-1}_{jk}\langle u_{k}(0), \L (\Q\L)^q\Q\L u_{i}(0)\rangle_{\rho}$ 
\\
MZ-Faber         & $\displaystyle e^{tc_0}\frac{J_q(2t\sqrt{-c_1})}{(\sqrt{-c_1})^q}$   &  $\displaystyle \sum_{k=1}^M G^{-1}_{jk}\langle u_{k}(0), \L F_q(\Q\L)\Q\L u_{i}(0)\rangle_{\rho}$
\end{tabular}
\end{minipage}
\caption{Series expansions of the Mori-Zwanzig memory kernel \eqref{SFD}. In this Table, $J_q$ is the $q$th Bessel function of the first kind, $c_0$ and $c_1$ are real numbers defining the recurrence 
relation of the Faber polynomials $\{F_0,\dots F_n\}$, $M$ is the number of phase space functions $u_i(\bm x)$ (quantities of interest), and $G^{-1}_{ij}$ is the inverse of the Gram matrix \eqref{gram}. See \cite{zhu2018faber} for additional details.}
\label{tab:1}
\end{table}

\subsection{Calculation of the MZ memory kernel from first principles}
\label{sec:recursive}
In this Section we develop a new algorithm to calculate 
the MZ memory kernel \eqref{SFD} based on first-principles, i.e., 
based on the microscopic equations of motion of the 
system \eqref{eqn:nonautonODE}. The 
algorithm we propose is built upon the combinatorial 
approach originally proposed by Amati, Meyer and 
Schilling in \cite{amati2018memory}. To illustrate the main idea
in a simple way, hereafter we study the case 
where the observable $\bm u(t)$ is one-dimensional, i.e., 
we have only one phase space function 
$u(t)=u(\bm x(t,\bm x_0))$. In this setting, the series 
expansion \eqref{K_formula} reduces to
\begin{equation}
K(t)\simeq \sum_{q=0}^n g_q(t) M_q,\qquad 
\text{where}\qquad M_q =\frac{
		\langle u(0), \L \Phi_q(\Q\L)\Q\L u(0)\rangle_{\rho}}{\langle u(0), u(0)\rangle_{\rho}}.
	\label{Mq}
\end{equation}
Note that $K(t)$ depends only on the set of parameters 
$\{M_0,\dots,M_n\}$, since the temporal modes $g_q(t)$ are 
explicitly available given the polynomial set 
$\{\Phi_0,\dots,\Phi_n\}$ (see Table \ref{tab:1}). 
We aim at determining  $\{M_0,\dots,M_n\}$ from first principles. 
For one-dimensional phase space functions, Mori's 
projection \eqref{Mori_P} reduces to 
\begin{equation}
\P f= \frac{\langle f  , u(0)\rangle_{\rho}}{\langle u(0), u(0)\rangle_{\rho}} u(0).
\end{equation}
At this point, it is convenient to introduce the following 
notation
\begin{equation}
\mu_i = \frac{\langle \L(\Q\L)^{i-1} u(0),u(0)\rangle_{\rho}}{\langle u(0),u(0)\rangle_{\rho}}, \qquad 
\gamma_i=\frac{\langle \L^i u(0),u(0)\rangle_{\rho}}{\langle u(0),u(0)\rangle_{\rho}}.
\label{mugam}
\end{equation}
Each coefficient $\mu_i$ represents the rescaling of $u(0)$ 
under the action of the operator $\P\L(\Q\L)^{i-1} $, i.e. we have 
\begin{equation}
\mu_i u(0) = \P\L(\Q\L)^{i-1} u(0).
\label{mu1}
\end{equation}
Clearly, if we are given $\{\mu_1,\dots,\mu_{n+2}\}$ 
then we can easily compute $M_q$ in \eqref{Mq}, and 
therefore the memory kernel $K(t)$ for any given 
polynomial function $\Phi_q$. For example, if 
$\Phi_q(\Q\L)=(\Q\L)^q$ then $M_{q}=\mu_{q+2}$ 
($q=0,\dots,n$). On the other hand, 
if $\{\Phi_0,\dots,\Phi_n\}$ are Faber polynomials \cite{zhu2018faber}, then 
we can write each $\Phi_q$ as a linear combination of 
monomials $(\Q\L)^j$ $(j=0,\dots, q)$  
and therefore represent $M_q$ as a linear 
combination of $\{\mu_1,\dots,\mu_{q+2}\}$.
Computing $\mu_i$ using the definition \eqref{mugam} 
involves taking operator powers and averaging, 
which may be computationally expensive. 
An alternative effective algorithm relies 
on the following recursive formula 
\cite{chu2017mori,snook2006langevin,berne1970calculation}  
\begin{align}
\mu_1=\gamma_1,\qquad 
\mu_2=\gamma_2-\mu_1\gamma_1,\qquad\cdots,\qquad 
\mu_{n}&=\gamma_{n}-\sum_{j=1}^{n-1} \mu_{n-j} \gamma_{j}.
\label{iterative1}
\end{align} 
In practice, \eqref{iterative1} shifts the problem of 
computing $\{\mu_1,\dots,\mu_{n}\}$ to the problem 
of evaluating the coefficients $\{\gamma_1,\dots,\gamma_{n}\}$
defined in \eqref{mugam}. This will be discussed extensively 
in the subsequent Section \ref{sec:calculation_of_gamma}. 
If the Liouville operator $\L$ is skew-adjoint 
relative to the inner product \eqref{ip}, then all $\mu_{j}$ 
and $\gamma_{j}$ corresponding to odd indices
are identically zero. This allows us to simplify the 
recursion \eqref{iterative1} as
\begin{align}
\label{iterative} 
\mu_{2j}&=\gamma_{2j}-\sum_{k=1}^{j-1} \mu_{2j-2k} 
\gamma_{2k} \qquad j=1,2, \dots .
\end{align} 
As a consequence, the streaming term \eqref{streaming} in 
the MZ equation vanishes identically since $\Omega=\mu_1=\gamma_1=0$.
We recall that skew-adjoint Liouville operators arise 
naturally, e.g., in Hamiltonian dynamical systems at statistical 
equilibrium.

\subsection{Systems with polynomial nonlinearities}
\label{sec:calculation_of_gamma}
In this Section, we address the problem of calculating the coefficients 
$\{\gamma_1,\dots,\gamma_{n}\}$ defined in \eqref{mugam} and 
appearing in the recursion relation \eqref{iterative1}. With such coefficients 
available, we can compute $\{\mu_1,\dots,\mu_{n}\}$ and therefore
the MZ memory kernel \eqref{Mq}. The calculation we propose 
is based on first principles, meaning that we do not rely on any 
assumption or model to evaluate the averages $\gamma_i=\langle \L^i u(0),u(0)\rangle_{\rho}/\langle u(0),u(0)\rangle_{\rho}$.
Instead, we develop a combinatorial algorithm that allows 
us to track all terms in $\L^iu(0)$, hence representing $\gamma_i$ 
exactly as a superimposition of a finite, although possibly large, 
number of terms. The algorithm we develop is built 
upon the combinatorial algorithm recently 
proposed by Amati, Meyer and Schiling in \cite{amati2018memory}. 
To describe the algorithm, consider the nonlinear dynamical system 
 \eqref{eqn:nonautonODE} and assume that $\bm F(\bm x)$ is 
a multivariate polynomial in the phase variables $\bm x$. 
A simple example of such system is the Kraichnan-Orszag 
three-mode problem \cite{Orszag,Xiao1,Brennan2018}
\begin{equation}
\dot{x}_1=x_1x_3, \qquad 
\dot{x}_2=-x_2x_3, \qquad 
\dot{x}_3=x_2^2-x_1^2.
\label{eqn:odexm}
\end{equation}
Other examples are the semi-discrete 
form of the Navier-Stokes equations, or the 
semi-discrete form of the nonlinear wave equation 
discussed in Section \ref{sec:application}. 
The key observation to compute $\gamma_j$ for systems with polynomial 
nonlinearities is that the action of the operator power $\L^i$ 
on a polynomial observable $u(\bm x)$ yields a polynomial function. 
For instance, consider $u(\bm x)=x_1^3$, and the  
Liouville operator associated with the system \eqref{eqn:odexm} 
\begin{equation}
\L= 
 x_1x_3\frac{\partial }{\partial x_1}
-x_2x_3\frac{\partial }{\partial x_2}
+(x_2^2-x_1^2)\frac{\partial }{\partial x_3}.
\label{LiouvilleKO}
\end{equation}
We have   
\begin{align}
\L x_1^3     = & 3 x_1^{3}x_3,  \label{Lpow1} \\ 
\L^2 x_1^3 = & 9 x_1^3 x_3^2    +  3 x_1^{3}x_2^2 - 3 x_1^{5}, \label{Lpow2}\\
\L^3 x_1^3 = & 27 x_1^3 x_3^3  + 18 x_1^3x_2^2 x_3 - 18 x_1^5 x_3 +
            27 x_1^{3}x_2^2x_3 - 6x_1^{3}x^2_2x_3-15 x_1^{5}x_3.  \label{Lpow3}
\end{align}
Clearly,  the number of terms in $\L^i x_1^3$ can rapidly 
increase, if high-order powers of $\L$ are considered. For higher-dimensional 
systems with non-local interactions, i.e., for systems where each $F_i(\bm x)$ 
($i=1,\dots, N$) depends on all components of $\bm x$, this 
problem is  serious, and requires multi-core 
computer-based combinatorics to systematically 
track all terms in the expansion 
of $\L^i x_j^q$. Let us introduce the following notation
\begin{align}\label{Ln}
\L^n x_j^{q}=\sum_{\bm b_i\in B^{(n)}}a^{(n)}_{\bm b_i} 
x_{k_{1}}^{m_{k_{1}}^{(i)}}\cdots x_{k_{r}}^{m_{k_{r}}^{(i)}},
\end{align}
where $\{a_{\bm b_i}^{(n)}\}$ are polynomial coefficients, and 
$\{m^{(i)}_{k_j}\}$ are polynomial exponents. The 
set of indexes representing the relevant phase phase variables 
appearing in $\L^n x_j^{q}$, i.e., $\{k_1,\dots, k_r\}$, 
is collected in the index set $K(n,j)=\{k_1,\dots, k_r\}$, which 
depends on $n$ and $j$. For example, in 
\eqref{Lpow1}-\eqref{Lpow3} we have
\begin{equation}
K(1,1)=\{1,3\}, \qquad K(2,1)=\{1,2,3\}, \qquad K(3,1)=\{1,2,3\}.
\end{equation}
Of course, for low-dimensional dynamical systems, the simplest choice 
for the relevant variables would be the complete set of variables 
$\{x_1,\cdots,x_N\}$. However, for high-dimensional systems with 
local interactions this choice could lead to unnecessary 
computations. In fact, it can be shown that the 
variables appearing in the polynomial $\L^n x_j^q$ are  
usually a (possibly small) subset of the phase variables 
if the system has local interactions. 
The vector $\bm b_i=[m^{(i)}_{k_{1}},\cdots,m^{(i)}_{k_{r}}]$, 
collects the exponents of the $i$-th monomial 
appearing in the expansion \eqref{Ln}.  
Similarly, $a^{(n)}_{\bm b_i}$ 
is the coefficient multiplying $i$-th monomial in \eqref{Ln}.
For example, in \eqref{Lpow1} and \eqref{Lpow2} we have, respectively,
\begin{equation}
\bm b_1=[3, 1], \qquad a^{(1)}_{\bm b_1}=3,\nonumber
\end{equation} 
\begin{align} 
\bm b_1=[3,0,2],\quad 
\bm b_2=[3,2,0],\quad 
\bm b_3=[5,0,0], \quad
a^{(2)}_{\bm b_1}=9,\quad 
a^{(2)}_{\bm b_2}=3, \quad 
a^{(2)}_{\bm b_3}=-3.
\nonumber
\end{align} 
At this point, it is convenient to define the set of polynomial exponents 
$B^{(n)}=\{\bm b_1,\bm b_2,\cdots \}$, the set polynomial 
coefficients $A^{(n)}=\{a^{(2)}_{\bm b_1},a^{(2)}_{\bm b_2},\cdots\}$, 
and the combined index set 
\begin{align}
\I^{(n)}=\{A^{(n)},B^{(n)}\}.
\label{def1}
\end{align}  
Clearly, $\I^{(n)}$ identifies uniquely the 
polynomial \eqref{Ln}, i.e., there is a one-to-one correspondence between 
 $\I^{(n)}$ and $\L^n x_j^q$. For example, in the case 
of \eqref{Lpow1}-\eqref{Lpow3} we have
\begin{align}
\I^{(1)}=&\{\underbrace{\{3\}}_{A^{(1)}}, 
\underbrace{\{[3,0,1]\}}_{B^{(1)}}\},\label{i1}\\
\I^{(2)}=&\{\underbrace{\{9,3,-3\}}_{A^{(2)}},
\underbrace{\{[3,0,2],[3,2,1],[5,0,0]\}}_{B^{(2)}},\label{i2}\\
\I^{(3)}=&\{\underbrace{\{27,18,-18,27,-6,-15\}}_{A^{(3)}},
\underbrace{\{[3,0,3],[3,2,1],[5,0,1],[3,2,1],[3,2,1],[5,0,1]\}}_{B^{(3)}}
\}.\label{i3}
\end{align}
If we apply $\L$ to \eqref{Ln} we obtain 
\begin{align}
\L^{n+1} x_j^{q}=&\L \L^{n} x_j^{q},\nonumber\\
= &\L \sum_{\bm b_i\in B^{(n)}}a^{(n)}_{\bm b_i} 
x_{k_{1}}^{m_{k_{1}}^{(i)}}\cdots x_{k_{r}}^{m_{k_{r}}^{(i)}},\nonumber\\
=&\sum_{\bm b_i\in B^{(n+1)}}a^{(n+1)}_{\bm b_i} 
x_{k_{1}}^{m_{k_{1}}^{(i)}}\cdots x_{k_{r}}^{m_{k_{r}}^{(i)}}.
\label{Ln1}
\end{align}
Clearly, if we can compute 
the mapping $\I^{(n)}\xrightarrow{\L} \I^{(n+1)}$, 
induced by the action of the Liouville operator $\L$ to 
the polynomial \eqref{Ln} (represented by $\I^{(n)}$), then  
we can compute the {\em exact} series expansion 
of $\L^n x_j^q$ for arbitrary $n$.
With such expansion available, we can immediately determine 
the coefficients $\gamma_j$ in \eqref{mugam} by averaging over the probability density $\rho$ as
\begin{equation}
\gamma_n=\frac{\langle\L^{n} x_j^{q},x_j^{q}\rangle_{\rho}}{\langle x_j^{q},x_j^{q}\rangle_{\rho}}=
\sum_{\bm b_i\in B^{(n)}}a^{(n)}_{\bm b_i}
\frac{\langle x_{k_{1}}^{m_{k_{1}}^{(i)}}\cdots 
x_{k_{r}}^{m_{k_{r}}^{(i)}}x_j^q \rangle_{\rho}}{\langle x_j^{q},x_j^{q}\rangle_{\rho}}.
\label{gamma_n}
\end{equation}
If the operator $\L$ is skew-adjoint in $L^2(\M,\rho)$, i.e., 
if $\langle \L f,g\rangle_{\rho}=-\langle f, \L g\rangle_{\rho}$, 
then we have
\begin{align}
\gamma_{2n}=\frac{\langle\L^{2n} x_j^{q},x_j^{q}\rangle_{\rho}}{\langle x_j^{q},x_j^{q}\rangle_{\rho}}
=(-1)^n\sum_{\substack{\bm b_i,\bm b_j\in B^{(n)}}}
a^{(n)}_{\bm b_j}a^{(n)}_{\bm b_i}\frac{\langle
x_{k_{1}}^{m_{k_{1}}^{(i)} +m_{k_{1}}^{(j)}} \cdots 
x_{k_{r}}^{ m_{k_{r}}^{(i)} + m_{k_{r}}^{(j)}}\rangle_{\rho}}{\langle x_j^{q},x_j^{q}\rangle_{\rho}}.
\label{gamma_2n}
\end{align}
{  As pointed out by Maiocchi {\em et al.} 
in \cite{maiocchi2012series}, the value 
of the first few coefficients $\{\gamma_n\}$ in 
\eqref{gamma_n} or \eqref{gamma_2n} can be used to 
identify non-exponential relaxation patterns to equilibrium.}

\vspace{0.3cm}
\noindent
\paragraph{Remark} To {  enhance} numerical 
stability when computing the Faber expansion of the MZ
memory kernel we scale the Liouville operator 
$\L$ by a parameter $\delta <1$  
(see \cite{zhu2018faber,huisinga1999faber}), i.e., we 
compute the Faber operator polynomial series 
relative to $\delta\L$.  Correspondingly, the 
generalized Langevin equation \eqref{gle} is 
solved on a time scale $t/\delta$. In this setting, the 
coefficients \eqref{gamma_n} are also calculated 
relative to the rescaled Liouville operator $\delta \L$.

\vspace{0.3cm}
\noindent
\paragraph{Remark} 
Computing $\gamma_j$ for linear systems reduces 
to a classical numerical linear algebra problem, i.e., 
computing the Rayleigh quotient of a matrix power. 
To show this, consider the $N$-dimensional linear 
system $\dot{\bm x}=\bm{A}\bm x$, evolving from 
the random initial state $\bm x_0\sim \rho_0$ 
($\bm x_0$ column vector). 
Suppose we are interested in the first component of the system, 
i.e., set the observable as $u(t)=x_1(t,\bm x_0)$. 
Define the linear subspace 
$V=\text{span}\{x_{01},x_{02},\cdots,x_{0N}\}\subset L^2(\M,\rho_0)$. 
Clearly $u(t)\in V$ for all $t\geq 0$ \cite{zhu2018estimation,zhu2018faber}. 
This allows us to calculate $\gamma_j$ as  
\begin{align}
\gamma_j=\langle \left[\bm {A}^T\right]^j\bm x_0 \cdot \bm e_1\rangle_{\rho_0},
\qquad j=1,\dots, n
\label{linear_gamman}
\end{align}
where $\bm e_1=[1,0,\dots,0]^T$.

\subsection{Mapping the index set $\I^{(n)}$}
\label{sec:iterativealgorithm}
Here we describe the algorithm  that allows 
us to compute the polynomial $\L^{n+1} x_k^q$ given 
the polynomial $\L^n x_k^q$ and the Liouville operator 
$\L$, i.e., the mapping defined by equation \eqref{Ln1}. 
This is equivalent to develop a set of algebraic rules 
to transform the combined index set $\I^{(n)}$ defined in 
\eqref{def1} into $\I^{(n+1)}$, for arbitrary $n$. 
Once such rules are known, we can apply them recursively 
to compute the polynomial sequence
\begin{equation}
x_j^q \rightarrow \L x_j^q\rightarrow \L^2 x_j^q\rightarrow   \L^3 x_j^q \rightarrow \cdots \rightarrow  \L^n x_j^q
\nonumber
\end{equation}
to any desired order. In this way, we can determine 
$\gamma_n$ through  \eqref{gamma_n} (or \eqref{gamma_2n}),  
$\mu_n$ through \eqref{iterative1} (or \eqref{iterative}), and 
therefore the MZ memory kernel \eqref{Mq}. 
Before formulating the algorithm in full generality, 
it is useful to examine how it operates in a concrete example. 
To this end, consider again the Kraichnan-Orszag 
system \eqref{eqn:odexm}, and the transformation 
between the polynomials \eqref{Lpow2} and \eqref{Lpow3}
defined by the action of the Liouville operator \eqref{LiouvilleKO}. 
We are interested in formulating such transformation 
in terms of a set of algebraic operations mapping 
the index set $\I^{(2)}$ into $\I^{(3)}$ 
(Eqs. \eqref{i2}-\eqref{i3}). We begin by decomposing 
the three-dimensional Liouville operator \eqref{LiouvilleKO} as 
\begin{equation}
\L = \L_1+\L_2+\L_3,\qquad \text{where}\qquad 
\L_1=x_1x_3\frac{\partial }{\partial x_1},\qquad 
\L_2=-x_2x_3\frac{\partial }{\partial x_2},\qquad 
\L_3=(x_2^2-x_1^2)\frac{\partial }{\partial x_3}.
\end{equation}
The action of $\L_i$ on any monomial 
generates a polynomial with $S_i$ terms. In the present example, 
we have $S_1=S_2=1$ and $S_3=2$. 
Let us now consider the first monomial in \eqref{Lpow2}, i.e., 
$9x_1^3x_3^2$. Such monomial is represented by the first element 
of $A^{(2)}$ and $B^{(2)}$ in \eqref{i3}. The corresponding 
combined set is $\{9,[3,0,2]\}$. At this point, we apply the operators 
$\L_1$, $\L_2$ and $\L_3$ to the polynomial $\{9,[3,0,2]\}$. This yields 
\begin{align}
\underbrace{\{9,[3,0,2]\}}_{9x_1^3x_3^2} &
\xrightarrow{\L_1} \underbrace{\{27,[3,0,3]\}}_{27x_1^3x_3^3},\\
\underbrace{\{9,[3,0,2]\}}_{9x_1^3x_3^2} &
\xrightarrow{\L_2} \underbrace{\{0,[3,-1,2]\}}_{0},\\
\underbrace{\{9,[3,0,2]\}}_{9x_1^3x_3^2} &
\xrightarrow{\L_3} 
\underbrace{\{18,[3,2,1]\}\, \biguplus\, \{-18,[5,0,1]\}}_{18 x_1^3x_2^2 x_3 - 18 x_1^5 x_3} = 
\{\{18,-18\},\{[3,2,1],[5,0,1]\}\}.
\label{u3}
\end{align}
The transformation associated with $\L_3$ generates 
the sum of two monomials, namely $18 x_1^3x_2^2 x_3 - 18 x_1^5 x_3$, 
which we represent as a union between two index sets. 
Such union, here denoted as $\biguplus$, is an ordered union 
that  pushes to the left polynomial coefficients 
and to the right polynomial exponents. Proceeding in a 
similar manner for all other monomials in \eqref{Lpow2} and 
taking ordered unions of all sets, yields the desired 
mapping $\I^{(2)}\rightarrow \I^{(3)}$. 
Let us now examine the action of a more general Liouville operator 
\begin{equation}
\L_j = z x_1^{c_1}\cdots x_N^{c_N}\frac{\partial }{\partial x_j}
\end{equation}
on the monomial $a x_1^{m_1}\cdots x_N^{m_N}$ 
represented by the index set $\{a,[m_1,\dots , m_N]\}$. 
We have 
\begin{equation}
\{a,[m_1, \dots ,m_N]\} \xrightarrow{\L_j} 
\{z m_j a, [m_1+c_1, \dots, m_j+c_j-1,\dots, m_N+c_N]\}. 
\end{equation}
This defines two {linear} transformations: a scaling transformation 
in the space of coefficients,  and an addition in the space of exponents
\begin{equation}
a\rightarrow (z m_1) a, \qquad [m_1, \dots ,m_N]
\rightarrow [m_1, \dots ,m_N] + [c_1, \dots, c_j-1,\dots, c_N].
\label{linT}
\end{equation}
In a vector notation, upon definition of 
$\bm b=[m_1, \dots ,m_N]$, $\bm \theta_j=[c_1, \dots, c_j-1,\dots, c_N]$ and $\alpha_j=z m_j$, we can write \eqref{linT} in 
compact form as  
\begin{equation}
 a\rightarrow \alpha_j a,\qquad  \bm b \rightarrow \bm b+\bm\theta_j.
 \label{linearMAP}
\end{equation}
Let us know consider the general case where the Liouville 
operator is defined as 
\begin{equation}
\L(\bm x)= \sum_{k=1}^N \L_k(\bm x) \qquad 
\L_k(\bm x)=F_{k}(\bm x)\frac{\partial }{\partial x_k} 
\end{equation}
and $F_k(\bm x)$ is a polynomial involving $S_k$ 
monomials in either all variables $\{x_1,\dots,x_N\}$ or 
a subset of them. The action of $\L$ on each 
monomial in \eqref{Ln1} can be written as
\begin{equation}
\L x^{m^{(i)}_{k_1}}_{k_1}\dots x^{m^{(i)}_{k_r}}_{k_r} = 
 \sum_{q\in K(n,j)} \L_q
x^{m^{(i)}_{k_1}}_{k_1}\dots x^{m^{(i)}_{k_r}}_{k_r},
\label{mono}
\end{equation}
where $K(n,j)=\{k_1,\dots, k_r\}$ is the set of relevant variables at iteration 
$n$. The polynomial \eqref{mono} involves 
$S_{k_1}+\cdots + S_{k_r}$ terms, each one of which can 
be explicitly constructed by applying the linear transformation 
rules \eqref{linearMAP}. In summary, we 
have
\begin{align}\label{iter_formula}
\I^{(n+1)}=\biguplus_{q \in K(n,j)}\biguplus_{i=1}^{\# B^{(n)}}
\biguplus_{s=1}^{S_{q}}\{\alpha^q_s a_{\bm b_i}^{(n)},\bm b_i+\bm 
\theta^q_s\}, 
\end{align}
where $\# B^{(n)}$ denotes the number of elements 
in $B^{(n)}$. Note that both $\alpha_s^s$ and $\bm \theta_s^q$ 
depend on $q \in K(n,j)$ (index set of relevant variables).

\paragraph{Remark} The recursive algorithm summarized by 
formula \eqref{iter_formula} is a modified version of 
the algorithm originally proposed by Amati, Meyer and Schiling 
in \cite{amati2018memory}.  {The key 
idea is the same, i.e., to compute the expansion 
coefficients $\gamma_n$ in 
\eqref{gamma_n} using polynomial differentiation. 
However, there are a few differences between our algorithm 
and the algorithm proposed in \cite{amati2018memory}
which we emphasize hereafter.} In \cite{amati2018memory}, 
the index set $B^{(n)}$ is pre-computed using the so-called 
spreading operators. Essentially, for each $n$, the iterative 
scheme generates a new set of polynomial coefficients 
$A^{(n)}$, which is subsequently matched 
with the corresponding indexes in $B^{(n)}$. 
In our algorithm, the sets $B^{(n)}$ and  
$A^{(n)}$ are computed on-the-fly 
at each step of the recursion. By doing so, we 
avoid calculating the spreading operators. This, in turn, 
allows us to avoid using numerical tensors to 
store index sets, since in our formulation 
there is no matching procedure between the 
polynomial exponents and the polynomial 
coefficients. {  Another difference 
between the two algorithms is that we utilized a rescaled 
Liouville operator $\delta \L$ ($\delta\in \mathbb{R}$) 
to enhance numerical stability when computing the 
operator polynomials $\Phi_q(\Q\L)$ in \eqref{Mdef}.  
The algorithm in \cite{amati2018memory}, on the other 
hand, is based on a Taylor series expansion of the operator 
exponential $e^{t\L}$, with unscaled Liouville 
operator\footnote{  
In our recent work \cite{zhu2018faber} (\S 3.1) we proved that 
a Taylor series of the orthogonal dynamical propagation 
$e^{t\Q\L}$ yields an expansion of the MZ memory integral 
that resembles the classical Dyson series in scattering theory.}}.

\subsection{An example: the Fermi-Pasta-Ulam (FPU) model}
Consider a one-dimensional chain of 
$N$ anharmonic oscillators with Hamiltonian
\begin{align}\label{FPU_H}
\H(\bm p,\bm q)=\sum_{j=0}^{N-1}\frac{p_j^2}{2m}+\sum_{j=1}^{N-1}V(q_j-q_{j-1}).
\end{align}
In \eqref{FPU_H} $\{q_j, p_j\}$ are, respectively, the generalized 
coordinate and momentum of the $j$-th oscillator, while  
$V(q_i-q_{i-1})$  is the potential energy between two adjacent 
oscillators. Suppose that the oscillator chain is closed (periodic), 
i.e., that $q_0=q_{N}$ and $p_0=p_{N}$. 
Define the distance between two oscillators 
as   $r_j=q_j-q_{j-1}$. This allows us to write the Hamilton's 
equations of motion  as 
\begin{align*}
\begin{dcases}
\frac{dr_j}{dt}&=\frac{1}{m}(p_i-p_{i-1}),\\
\frac{dp_j}{dt}&=\frac{\partial V(r_{j+1})}{\partial r_{j+1}}-
\frac{\partial V(r_{j})}{\partial r_{j}}.
\end{dcases}
\end{align*}
The Liouville operator corresponding to this system is 
\begin{align*}
\L(\bm p,\bm r)=\sum_{i=1}^{N-1}\left[
\left(\frac{\partial V(r_{i+1})}{\partial r_{i+1}}
-\frac{\partial V(r_i)}{\partial r_i}\right)\frac{\partial}{\partial p_i}+\frac{1}{m}(p_i-p_{i-1})\frac{\partial}{\partial r_i}\right].
\end{align*}
Setting $V(x)=\alpha x^2/2+\beta x^4/4$ yields 
the well-known Fermi-Pasta-Ulam $\beta$-model 
\cite{mendl2015current}, which we 
study hereafter.
To this end, suppose we are interested in the distance 
between the oscillators $j$ and $j-1$, i.e., in the polynomial 
observable $u(\bm p,\bm r)=r_j$. The action of 
$\L^n$ on $r_j$ can be explicitly written as  
\begin{align}\label{poly_form}
\L^nr_j=
\sum_{\bm b_{i}\in B^{(n)}}a_{\bm b_{i}}^{(n)}
r_{k_{1}}^{m_{k_{1}}^{(i)}}\cdots r_{k_{u}}^{m^{(i)}_{k_{u}}}
p_{l_{1}}^{s^{(i)}_{l_{1}}}\cdots p_{l_{v}}^{s^{(i)}_{l_{v}}},
\end{align}
where $\{k_{1},\dots,k_{u}\}$ and $\{l_{1},\dots,l_{v}\}$ are 
the relevant degrees of freedom for the polynomials of $\bm r$ and $\bm p$, respectively, at iteration $n$. 
We can explicitly compute the sets of such 
relevant degrees of freedom as  
\begin{align}
K_r(n,j)=\left\{j-\left\lfloor\frac{n}{2}\right\rfloor,\dots,j+\left\lfloor\frac{n}{2}\right\rfloor\right\}\quad
L_p(n,j)=\left\{j-\left\lfloor\frac{n+1}{2}\right\rfloor,\dots,j+\left\lfloor\frac{n-1}{2}\right\rfloor\right\},
\end{align}
The action of the Liouville operator on each monomial
appearing in \eqref{poly_form} can be written as 
\begin{align}\label{L_iteration}
\L 
r_{k_{1}}^{m^{(i)}_{k_{u}}} r_{k_{u}}^{m^{(i)}_{k_{u}}} 
p_{l_{1}}^{s^{(i)}_{l_{1}}} \cdots p_{l_{v}}^{s^{(i)}_{l_{v}}} 
=\sum_{v\in K_r(n,j)} \sum_{h\in L_p(n,j)} (\L_{r_{v}}+\L_{p_{h}})
r_{k_{1}}^{m^{(i)}_{k_{1}}}\cdots r_{k_{u}}^{m^{(i)}_{k_{u}}}
p_{l_{1}}^{s^{(i)}_{l_{1}}}\cdots p_{l_{v}}^{s^{(i)}_{l_{v}}},
\end{align} 
where 
\begin{equation}
\L_{r_{v}} =\frac{1}{m}(p_v-p_{v-1})\frac{\partial}{\partial r_v},
 \qquad \text{and}\qquad 
 \L_{p_{h}}=\left[ \alpha(r_{h+1}-r_{h})+\beta\left(r_{h+1}^3-r_{h}^3\right)\right]\frac{\partial}{\partial p_h}.
\end{equation}
By computing the action of $\L_{r_{v}}$ 
and $\L_{p_{h}}$ on the monomial 
$r_{k_{1}}^{m^{(i)}_{k_{1}}}\cdots r_{k_{u}}^{m^{(i)}_{k_{u}}}
p_{l_{1}}^{s^{(i)}_{l_{1}}}\cdots _{l_{v}}^{s^{(i)}_{l_{v}}}$ 
we obtain explicit linear 
maps of the form \eqref{linearMAP}, involving the polynomial exponents 
\begin{equation}
\bm b_i=[\bm m^{(i)},\bm s^{(i)}],\qquad 
\bm m^{(i)}=[m^{(i)}_{k_1}, \dots ,m^{(i)}_{k_u}],\qquad 
\bm s^{(i)}= [s^{(i)}_{l_1},\dots,s^{(i)}_{l_v}],
\end{equation}
and the polynomial  coefficients 
$a_{\bm b_i}^{(n)}$.
With such maps available, we can transform the 
combined index set $\I^{(n)}$ (representing $\L^n r_j$)
to $\I^{(n+1)}$ (representing $\L^{n+1}r_j$) 
using \eqref{iter_formula}. Specifically, we obtain
\begin{align*}
\I^{(n+1)}=\I^{(n+1)}_r \biguplus \I^{(n+1)}_p,
\end{align*}
where 
\begin{align*}
\I^{(n+1)}_r&=\biguplus_{v\in K_r(n,j)}\biguplus_{i=1}^{\# B^{(n)}}
\biguplus_{k=0}^{1}\left\{m_v^{(i)}(-1)^ka_{\bm b_i}^{(n)},
[\bm m^{(i)}-\bm e_v,\bm s^{(i)}+\bm e_{v-k}]\right\},\\
 \I^{(n+1)}_p&=\biguplus_{h\in L_p(n,j)}\biguplus_{i=1}^{\# B^{(n)}}
\biguplus_{k=0}^{1}\left\{\{s_{h}^{(i)}(-1)^{k+1}
\alpha a_{\bm b_i}^{(n)},s_{h}^{(i)}(-1)^{k+1}
\beta a_{\bm b_i}^{(n)} \},\right.\\
& \hspace{3.3cm}\left.\{[\bm m^{(i)}+\bm e_{h+k},\bm s^{(i)}-\bm e_{h}],
[\bm m^{(i)}+3\bm e_{h+k},\bm s^{(i)}-\bm e_{h}]\}\right\}.
\end{align*}

\section{ Modeling the MZ fluctuation term}
\label{sec:Model}
 {
In previous Sections we discussed an algorithm 
to approximate the memory 
kernel in the MZ equation \eqref{gle_full} or \eqref{gle} based 
on the microscopic equations of motion (first-principle calculation).
In this Section we construct a statistical model for fluctuation term 
$\bm f(t)$ appearing in \eqref{gle_full}, 
which will allow us to compute statistical properties 
of the quantity of interest $\bm u(t)$ 
beyond two-point correlations. A possible way 
to build such model is to expand \eqref{f} in a 
finite-dimensional series}\footnote{Note 
that $\bm f(t)$ is a random process obtained by mapping the 
random initial state $\bm u(0)=\bm u(\bm x_0)$ forward in time 
using the orthogonal dynamics propagator $e^{t\Q\L(\bm x_0)}$.} 
 {(see Eq. \eqref{generalS}) as
\begin{equation}
\bm f(t) \simeq \sum_{q=0}^n g_q(t)\Phi_q(\Q\L)\Q\L \bm u(0),
\label{fa}
\end{equation}
and evaluate the coefficients $\Phi_q(\Q\L)\Q\L \bm u(0)$ 
using the combinatorial approach discussed in Section \ref{sec:calculation_of_gamma}. 
However, this may not be straightforward since $\Phi_q(\Q\L)\Q\L \bm u(0)$ is a high-dimensional random field. }
An alternative approach is to ignore the mathematical 
structure of $\bm f(t)$, i.e., equation 
\eqref{f} or the series expansions \eqref{fa}, 
and simply model $\bm f(t)$ as a stochastic 
process. In doing so, we need to make sure that 
the statistical properties of the reduced-order model, e.g., 
the equilibrium distribution, are consistent with the 
full model. Such consistency conditions carry over 
a certain number of constraints on $\bm f(t)$, 
which allow for its {\em partial} identification. As an example, 
consider the following MZ model recently 
proposed by Lei {\em et al.} in \cite{lei2016data} to 
study the dynamics of a tagged particle 
in a large particle system
\begin{align}
\begin{dcases}
\dot{\bm{q}}&=\frac{\bm {p}}{m}\\
\dot{\bm{p}}&=\bm {F}(\bm{q})+\bm{d}\\
\dot{\bm{d}}&=\bm{B_0}\bm{d}-\bm{A}_0\frac{\bm{p}}{m}+\bm{f}(t)\\
\end{dcases}\label{lei_gle}
\end{align}
It was shown in \cite{lei2016data} that if $\bm f(t)$ is 
Gaussian white noise (in time) with auto-correlation 
function 
\begin{equation}
\langle \bm {f}(t)\bm{f}(t')\rangle=
-\beta \left(\bm{B}_0 \bm{A}_0+
\bm {A}_0\bm {B}_0^T\right)\delta(t-t'),
\end{equation} 
then the equilibrium distribution
of the particle system has the Boltzmann-Gibbs form
\begin{align}
\rho(\bm{p},\bm{q},\bm{d})\propto 
\exp\left\{-\beta\left(\frac{1}{2m}\left| \bm{p}\right|^2+\frac{1}{2}\bm{d}^T\bm{A}_0^{-1}\bm{d}+
V(\bm{q})\right)\right\},
\label{eqU}
\end{align} 
$V(\bm q)$ being the inter-particle potential. 
However, modeling $\bm f(t)$ as a Gaussian  process
does not provide satisfactory statistics in MZ equations 
is we use Mori's projection. In fact, 
equation \eqref{gle_full} is linear and therefore 
the equilibrium distribution of $\bm u(t)$ 
(assuming it exists) under Gaussian noise $\bm f(t)$ 
will be necessarily Gaussian. In most applications, however, 
the equilibrium distribution of $\bm u(t)$ is strongly 
non-Gaussian. To overcome this difficulty 
Chu and Li  \cite{chu2017mori} recently developed a 
multiplicative Gaussian noise model that generalizes 
\eqref{gle_full} in the sense that it is not based on 
additive noise, and it allows for non-Gaussian 
responses.

In this Section we propose a different stochastic modeling 
approach for $\bm f(t)$ based on bi-orthogonal representations 
random processes \cite{Venturi2,Venturi1,Venturi,Aubry,Aubry1}. 
To describe the method, we study the case where the observable  
$\bm u(t)$ is real valued (one-dimensional) and 
square integrable. This allows us to develop the 
theory in a clear and concise way. We also assume that the 
system is in statistical equilibrium, i.e., that there exists an 
equilibrium distribution $\rho_{eq}(\bm x)$ (or more generally 
an invariant measure) for the phase variables 
$\bm x(t,\bm x_0)$, and that $\bm x_0$ is 
sampled from such distribution. 
The MZ equation \eqref{gle_full} for one-dimensional 
phase space functions $u(t)=u(\bm x(t,\bm x_0))$ 
reduces to 
\begin{equation}
\frac{du(t)}{dt} = \Omega u(t) + \int_0^t K(t-s)u(s)ds +f(t), 
\label{gle_full_1D}
\end{equation}
where 
\begin{equation}
\Omega = \frac{\langle u(0), \L u(0)\rangle_{eq}}{\langle u(0), u(0)\rangle_{eq}},
\qquad 
K(t)=\frac{\langle u(0), \L f(t) \rangle_{eq}}{\langle u(0), u(0)\rangle_{eq}},
\qquad
f(t)=e^{t\Q\L}\Q\L u(0).
\label{OmegaKf}
\end{equation}
Since $u(t)$ is assumed to be a second-order random 
process in the time interval $[0,T]$, we can expand it in a 
truncated Karhunen-Lo\'{e}ve series
\begin{align}
u(t)\simeq \overline{u}+\sum_{k=1}^{K}\sqrt{\lambda_k}\xi_k e_k(t), \qquad t\in[0,T]
\label{KL_Sample}
\end{align}
where $\overline{u}$ denotes the mean of $u(t)$ relative to the 
equilibrium distribution\footnote{The mean 
of $u(t)=u(\bm x(t,\bm x_0))$ is necessarily time-independent
 at statistical equilibrium. In fact, at equilibrium we have 
that $\bm x_0\sim \rho_{eq}$ implies 
that $\bm x(t) \sim \rho_{eq}$ for all $t\geq 0$. 
A statistically stationary process however, may not be stationary in phase 
space. Indeed, $\bm x(t)$ evolves in time, 
eventually in a chaotic way as it happens for 
systems with strange attractors  and invariant 
measures.}, 
$\{\xi_1,\dots, \xi_K\}$ are uncorrelated random 
variables ($\langle\xi_i\xi_j\rangle_{eq}=\delta_{ij}$), 
and $\{\lambda_k,e_k(t)\}$ ($k=1,\dots,K$) are, respectively, 
eigenvalues and eigenfunctions of the homogeneous Fredholm integral 
equation of the second kind
\begin{equation}
\int_{0}^T\langle u(t)u(s)\rangle_{eq} e_k(s)ds=\lambda_k e_k(t), 
\qquad t\in[0,T].
\label{KLeigen}
\end{equation}
We recall that for ergodic systems in statistical equilibrium the 
auto-correlation function $\langle u(t)u(s)\rangle_{eq}$ decays to 
zero as $|t-s|\rightarrow \infty$. 
Also, the integral operator at the left hand side 
of \eqref{KLeigen} is positive-definite and 
compact \cite{Aubry1}. 
The orthogonal random variables 
$\xi_k$ and the temporal modes $e_k(t)$ are 
related to each other by the following dispersion 
relations \cite{Aubry,Venturi2}
\begin{equation}
\xi_k=\frac{1}{\sqrt{\lambda_k}}\int_{0}^T u(t)e_k(t)dt,\qquad  
e_k(t)=\frac{\langle u(t)\xi_k\rangle_{eq}}{\sqrt{\lambda_k}} \qquad 
k=1, 2, \dots.
\label{dispersion}
\end{equation}
Equation \eqref{dispersion} suggests that if $u(t)$ is 
a Gaussian random process (e.g., a Wiener process) 
then $\{\xi_1,\dots,\xi_K\}$ are necessarily 
independent Gaussian random variables. 
On the other hand, if $u(t)$ is  
non-Gaussian then the joint distribution of
$\{\xi_1, \dots,\xi_K\}$ is unknown, although  
it can be in principle computed  
by using the transformation $u(t) \rightarrow \xi_k$ ($k=1,..,K$) 
defined in \eqref{dispersion}, given  
$\lambda_k$ and $e_k(t)$. 

An alternative approach to identify the PDF of 
$\{\xi_1,\dots,\xi_K\}$ relies on sampling. 
In particular, as recently shown by Phoon {\em et al.} 
\cite{phoon2002simulation,phoon2005simulation}, it is possible 
to develop effective sampling algorithms for the 
KL expansion \eqref{KL_Sample}. 
Such algorithms allow to sample the 
uncorrelated variables $\{\xi_1,\dots,\xi_K\}$ 
in a way that makes the PDF of $u(t)$ 
consistent with the equilibrium distribution, 
which can be calculated by mapping 
$\bm x_0\sim \rho_{eq}(\bm x_0)$ to $u(\bm x_0)$.
At this point, we have available a consistent bi-orthogonal 
representation of the random process $u(t)$ defined by the series 
expansion \eqref{KL_Sample}. It is straightforward to see that 
such representation yields a corresponding series expansion 
of the fluctuation term $f(t)$ in \eqref{gle_full_1D}. 
In fact we have the following

\begin{prop}\label{prop:prop1}
For any bi-orthogonal series expansion 
\eqref{KL_Sample} of the solution to the MZ-equation 
\eqref{gle_full_1D}, there exists a unique series expansion 
of the fluctuation term $f(t)$ of the form 
\begin{equation}
f(t)=\overline{f}+\sum_{k=1}^{K} \sqrt{\lambda_k} \xi_k h_k(t).
\label{noise_expansion}
\end{equation}
\end{prop}

\begin{proof}
It is sufficient to prove the theorem for zero-mean 
processes. To this end, we set $\overline{u}=0$ 
and $\overline{f}=0$ in \eqref{KL_Sample} and
\eqref{noise_expansion}. A substitution 
of \eqref{KL_Sample} into \eqref{gle_full_1D} 
yields, for all $t\in [0,T]$
\begin{align}
f(t) = & \sum_{k=1}^{K} 
\sqrt{\lambda_k}\xi_k \left(\frac{de_k(t)}{dt} - \Omega 
e_k(t)-\int_0^t K(t-s)e_k(s)ds\right).\label{seriesf1}
\end{align}
Define, 
\begin{align}\label{eqn_1}
h_k(t) = \frac{de_k(t)}{dt} - \Omega e_k(t)-\int_0^t K(t-s)e_k(s)ds.
\end{align}
This equation does not allow us to compute $h_k$ explicitly quite yet. 
In fact, the MZ memory kernel $K(t-s)$ depends on $f(t)$ 
(see Eq. \eqref{OmegaKf}). However, a substitution of \eqref{noise_expansion} (with $\overline{f}=0$) 
into the analytical expression of $K(t)$ yields 
\begin{align}
K(t)=\sum_{i,j=1}^K \sqrt{\lambda_i\lambda_j}  v_{ij}e_i(0)h_j(t),\qquad \text{where}\qquad 
v_{ij}=\frac{\langle \xi_i,\L \xi_j\rangle_{eq}}{\langle u(0), u(0)\rangle_{eq}}. 
\label{kernel}
\end{align}
To evaluate $\langle \xi_i,\L \xi_j\rangle_{eq}$ we need to express 
$\{\xi_1,\dots,\xi_K\}$ as a function of $\bm x_0$ 
(recall that $\L$ operates on functions of $\bm x_0$,  
see Eq. \eqref{Koopman}), and then integrate 
over $\rho_{eq}(\bm x_0)$. This is easily achieved 
by using the dispersion relation \eqref{dispersion}. Specifically, 
we have
\begin{equation}
\xi_k(\bm x_0)=\frac{1}{\sqrt{\lambda_k}}\int_{0}^T
 u(\bm x(t,\bm x_0))e_k(t)dt.
\end{equation}
At this point, we substitute \eqref{kernel} into \eqref{eqn_1} to obtain
\begin{align}
h_k(t) = \frac{de_k(t)}{dt} - \Omega e_k(t)-\sum_{i,j=1}^K 
\sqrt{\lambda_i\lambda_j}   v_{ij}
e_i(0) \int_{0}^t h_j(t-s)e_k(s)ds.
\label{hmodes}
\end{align}
Given $\{e_1(t),\dots,e_K(t)\}$, $\Omega$ and $v_{ij}$, this 
equation can be solved uniquely for $\{h_1(t),\dots,h_K(t)\}$ 
by using Laplace transforms. Note that
$\{h_1(t),\dots,h_K(t)\}$ are not necessarily 
orthogonal in $L^2([0,T])$.

\end{proof}

\paragraph{Remark} If the dynamical 
system \eqref{eqn:nonautonODE} is Hamiltonian 
then the MZ steaming term vanishes, and 
the MZ memory kernel  can be written in terms the fluctuation 
term  as (see Eq. \eqref{2nd_FDT}) 
\begin{equation}
K(t)=\frac{\langle f(0),f(t)\rangle_{eq}}{\langle u(0),u(0)\rangle_{eq}}.
\end{equation}
A substitution of this expression into  
\eqref{gle_full_1D} yields, after projection onto $\xi_k$
\begin{align}\label{dd1}
\frac{de_k(t)}{dt}=\int_0^t\sum_{j=1}^K \lambda_{j}\left[h_{j}(0)h_{k'}(t-s)\right]e_k(s)ds+h_k(t).
\end{align}
This equation establishes a one-to-one correspondence between 
the temporal modes of the KL expansion \eqref{KL_Sample}
and the temporal modes of the fluctuation term \eqref{seriesf1}. 
In particular, given $\{e_1(t),\dots,e_K(t)\}$, we can 
determine $\{h_1(t),\dots,h_K(t)\}$ directly by using
Laplace transforms, without building the MZ memory 
kernel \eqref{kernel}.

\subsection{Building MZ-KL stochastic models from first principles}
Proposition \ref{prop:prop1} establishes a one-to-one 
correspondence between the noise process in the MZ 
equation \eqref{gle_full_1D} and the biorthogonal series 
expansion of the solution. This new paradigm allows us  
to build stochastic models for the observable $u(t)$ 
at statistical equilibrium from first principles. 
To this end, 
\begin{enumerate}
\item  Compute the solution to the MZ equation 
for the temporal correlation function of $u(t)$ (see Eq. \eqref{gle_C})
\begin{equation}
\frac{dC(t)}{dt}=\Omega C(t) + \int_{0}^t 
K(t-s)C(s) ds.
\label{gle_C_1D}
\end{equation}
The memory kernel $K(t-s)$ can be expanded as in 
\eqref{Mq}, and computed from first-principles using the 
combinatorial approach we discussed in Section \ref{sec:calculation_of_gamma}.

\item Build the Karhunen-Lo\`eve expansion 
\eqref{KL_Sample} by spectrally decomposing the 
correlation function $C(t)=\langle u(0)u(t)\rangle_{eq}$ 
 obtained at point 1. Recall that at statistical 
equilibrium we have 
$C(t-s)=\langle u(0)u(t-s)\rangle_{eq}=\langle u(s)u(t)\rangle_{eq}$.
This yields eigenvalues $\{\lambda_j\}$ and the eigenfunctions 
$e_j(t)$. The uncorrelated random variables $\xi_k$ appearing 
in \eqref{KL_Sample} can be sampled consistently with 
the equilibrium distribution $\rho_{eq}$ by using, e.g., the 
iterative algorithm recently proposed by Phoon 
{\em et al.} \cite{phoon2005simulation,phoon2002simulation}.

\item With $\{\xi_1,\dots,\xi_K\}$, $\{e_1(t),\dots,e_K(t)\}$ 
and  $\{\lambda_1,\dots,\lambda_K\}$ available, we 
can uniquely identify the noise process $f(t)$ in 
the MZ equation \eqref{gle_full_1D}. To this end, we simply 
use Proposition \ref{prop:prop1}, with the temporal 
modes $h_{k}(t)$ obtained by solving 
equation \eqref{hmodes} or \eqref{dd1} with the Laplace transform.

\item {  With $K(t)$ computed from first principles, 
and $f(t)$ modeled based on the auto-correlation 
function $C(t)$,  we can generate samples of the 
observable $u(t)$ by solving equation \eqref{gle_full_1D}.} 

\end{enumerate}

\paragraph{Remark}
{  We emphasize that the correlation function $C(t)$ 
can be also computed directly from data, e.g., by using a Monte-Carlo or 
a quasi Monte Carlo method \cite{qmc}. 
With $C(t)$ available it is possible to determine the 
fluctuation term $f(t)$ with equation \eqref{noise_expansion}
and the MZ memory kernel $K(t)$ using equation  \eqref{kernel}.}\\

The results of this Section 
can be generalized to vector-valued phase space functions 
$\bm u(t)$ at statistical equilibrium. The starting point is 
the KL expansion for multi-correlated stochastic processes 
we recently proposed in \cite{Cho2013}. Such expansion is 
constructed based on cross-correlation 
information\footnote{At statistical equilibrium the 
cross correlation functions are invariant under temporal shifts. 
This means that $\langle u_{i}(s),u_{j}(t)\rangle_{eq}=
\langle u_{i}(0),u_{j}(t-s)\rangle_{eq}$ for all $t\geq s$. 
Hence, the solution to the projected MZ 
equation \eqref{gle_C} is sufficient to compute the KL expansion 
of the multi-correlated process $\bm u(t)$, e.g., using the 
series expansion method proposed in \cite{Cho2013}.}, and 
can be made consistent with the equilibrium distribution of 
$\bm u(t)$, e.g., by using the sampling strategy proposed in
\cite{phoon2005simulation,phoon2002simulation}.
The correspondence between the KL expansions 
of $\bm u(t)$ and the vector-valued fluctuation 
term $\bm f(t)$ can be established
by following the same arguments we used in the 
proof of Proposition \ref{prop:prop1}.

\section{Applications to nonlinear systems with local interactions}
\label{sec:application}
In this Section,  we demonstrate the accuracy of 
the MZ memory calculation method and the 
reduced-order stochastic modeling technique we 
discussed in Section \ref{sec:recursive} and 
Section \ref{sec:Model}, respectively. 
To this end, we study nonlinear random wave 
propagation described by Hamiltonian partial 
differential equations (PDEs). To derive such PDEs  
consider the nonlinear functional
\begin{align}\label{Hamiltonian_NLW}
\H([p],[u])=\int_{0}^{2\pi}\left[\frac{p^2}{2}+\frac{\alpha}{2}
u_x^2 +G(p,u_x, u)\right]dx,
\end{align}
where $u=u(x,t)$ represents the 
wave displacement, $p=p(x,t)$ is the canonical 
momentum (field variable 
conjugate to $u(x,t)$), $u_x=\partial u/\partial x$, 
and $G(p, u_x,u)$ is the nonlinear interaction term. 
By taking functional derivatives of \eqref{Hamiltonian_NLW} with respect to $p$ and $u$ (see, e.g., \cite{venturi2018numerical}) 
we obtain the Hamilton's equations of motion
\begin{align}
\begin{dcases}
\partial_t u&=\frac{\delta\H(p,u)}{\delta p(x,t)}=p+\partial_pG(p, u_x,u),\\
\partial_t p&=-\frac{\delta\H(p,u)}{\delta u(x,t)}=\alpha  u_{xx} + \partial_{x}\partial_{u_{x}}G(p,u_x,u)-\partial_uG(p,u_x ,u).
\end{dcases}
\label{eq:wavesystem}
\end{align}
The corresponding nonlinear wave equation is 
\begin{align}\label{eqn:NLWE}
u_{tt}=\alpha u_{xx}+\partial_t\partial_pG(p, u_x,u)+
\partial_{x}\partial_{u_{x}}G(p,u_x,u)-\partial_uG(p,u_x ,u).
\end{align}
This equation has been studied extensively in mathematical 
physics 
\cite{klainerman1980global,donninger2012stable,mc1994statistical,parisi1988statistical}, 
in particular in general relativity, statistical mechanics, 
and in the theory of viscoelastic fluids.  
In Figure \ref{fig:NLW_flow} and Figure \ref{fig:NLW_path}, we 
plot a few sample numerical solutions to \eqref{eqn:NLWE} 
corresponding to different initial conditions and different 
nonlinear interaction term $G(p, u_x , u)$. 
These solutions are computed by an accurate Fourier spectral 
method with $N=512$ modes (periodic boundary conditions 
in $x\in[0,2\pi]$).
Throughout this Section, we assume that the initial 
state $\{u(x,0), p(x,0)\}$ is random and distributed 
according to the functional Boltzmann-Gibbs equilibrium 
distribution\footnote{The partition 
function $Z(\alpha,\gamma)$ is defined 
as a functional integral
over $u(x)$ and $p(x)$ (see, e.g., \cite{venturi2018numerical}).}
\begin{align}\label{equlibrium}
\rho_{eq}([p],[u]) = \frac{1}{Z(\alpha,\gamma)}
e^{-\gamma\H([p],[u])},
\quad \text{where}\qquad 
Z(\alpha,\gamma)=\int e^{-\gamma\H(p,u)}\D[p(x)]\D[u(x)].
\end{align}
We emphasize that $\rho_{eq}([p],[u])$ is invariant 
under the infinite-dimensional flow generated 
by \eqref{eqn:NLWE} with periodic boundary conditions, 
since the Hamiltonian \eqref{Hamiltonian_NLW} 
is a constant of motion (conserved quantity) in this case.
\begin{figure}[ht!]
\centering
\centerline{\hspace{0.4cm}
${G=0}$\hspace{5.5cm}
${G=\beta u_x^4/4}$
}

\centerline{
\includegraphics[height=5.0cm]{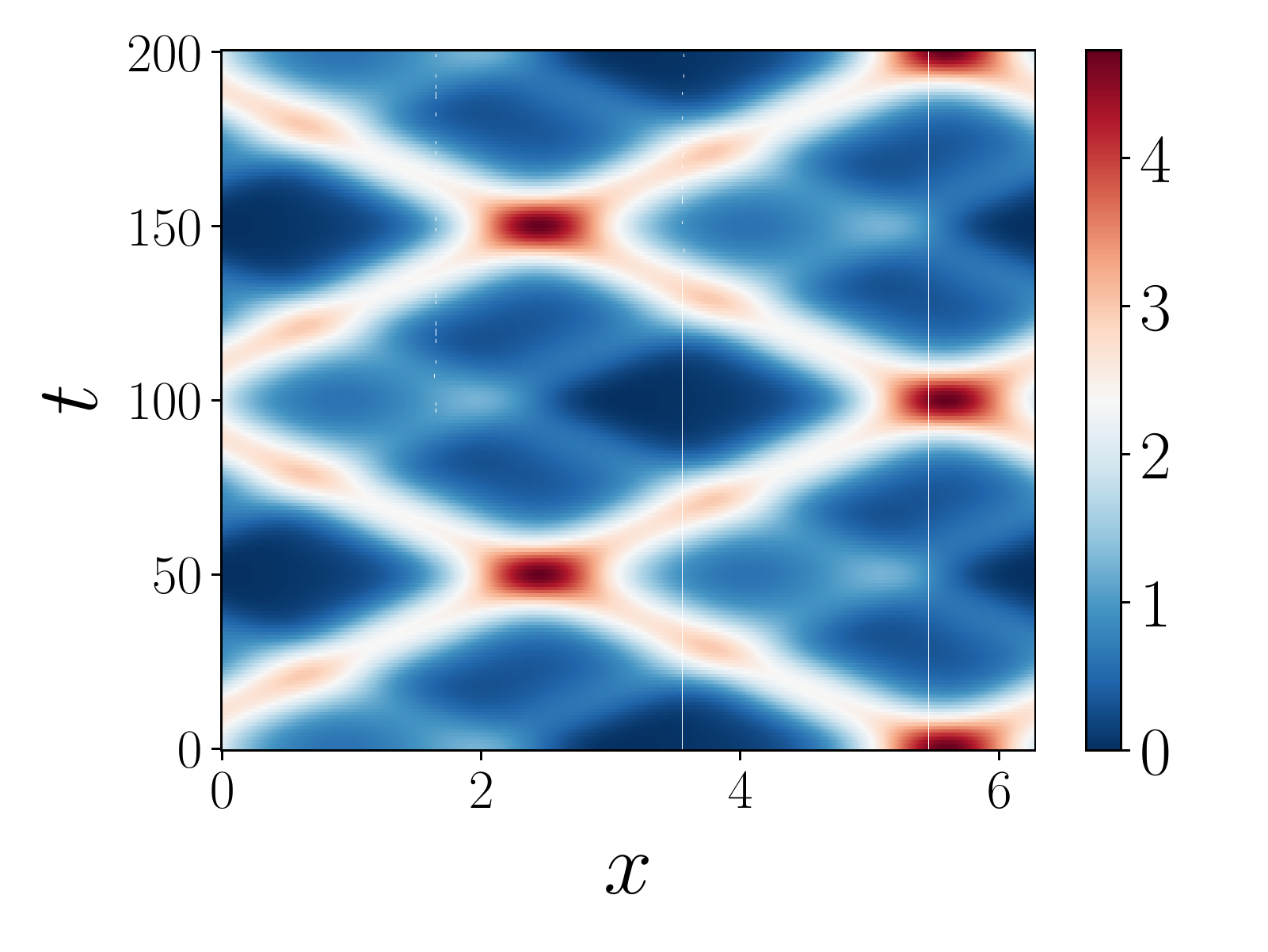} 
\includegraphics[height=5.0cm]{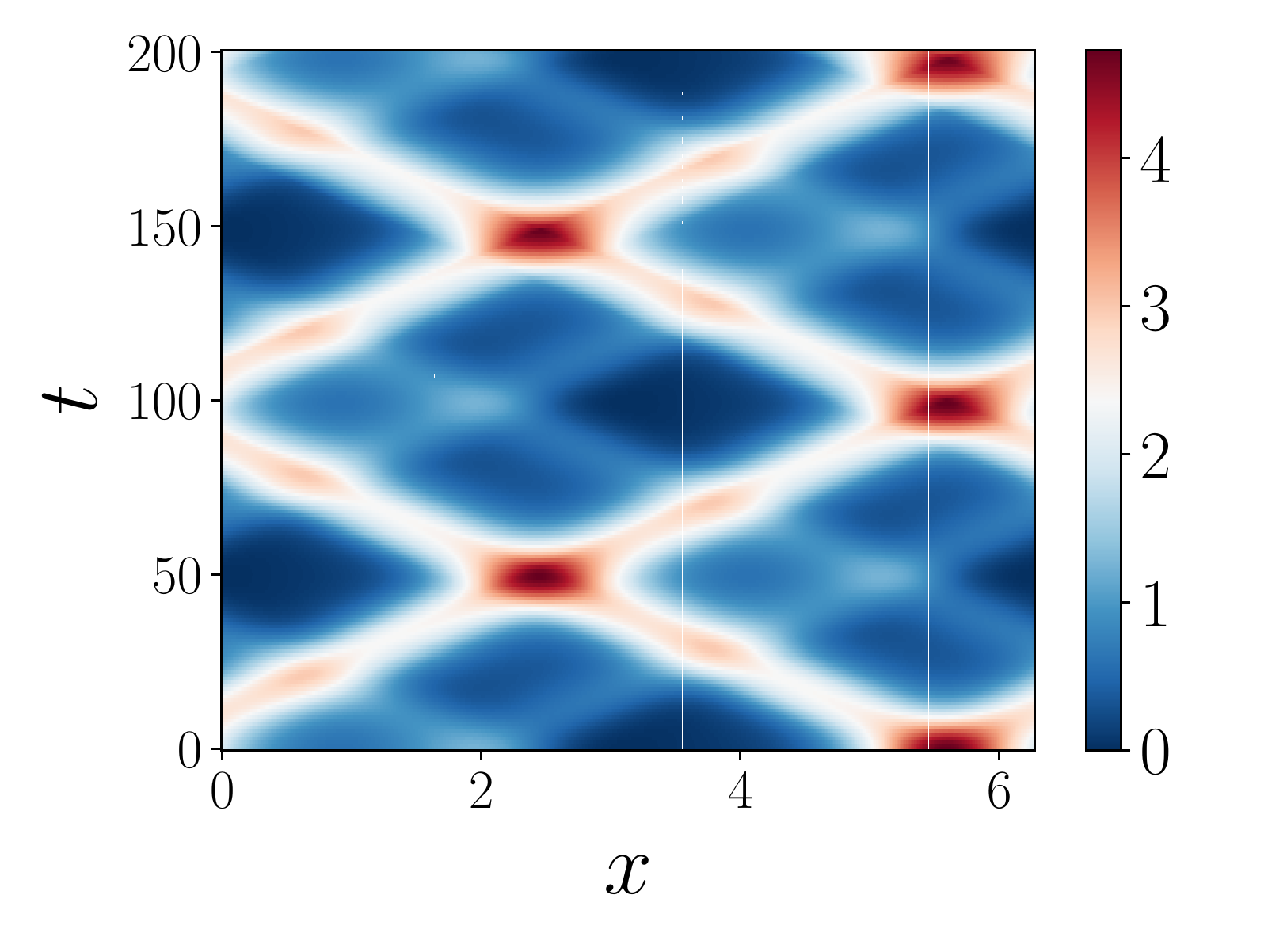}
}
\centerline{
\includegraphics[height=5.0cm]{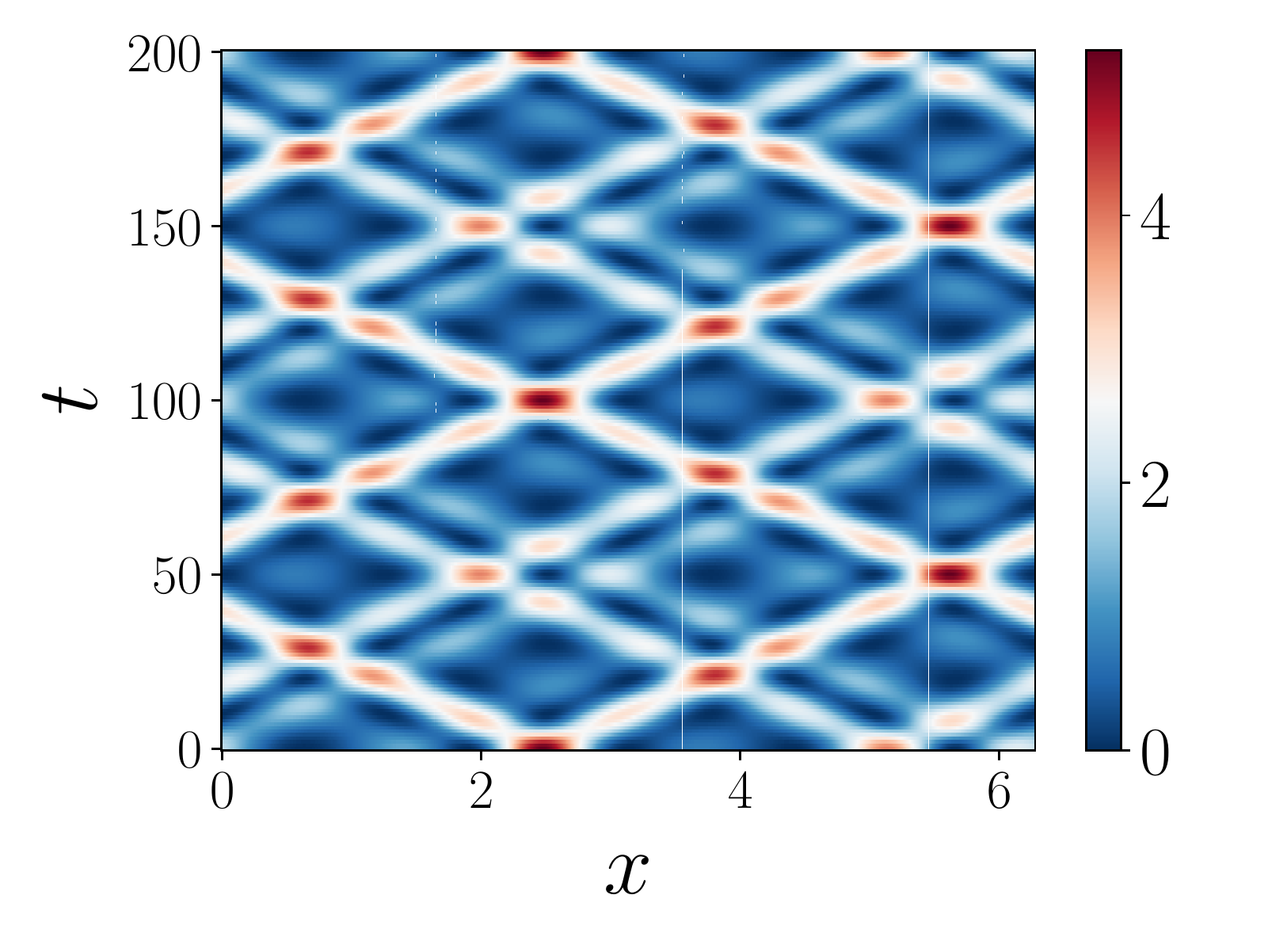}
\includegraphics[height=5.0cm]{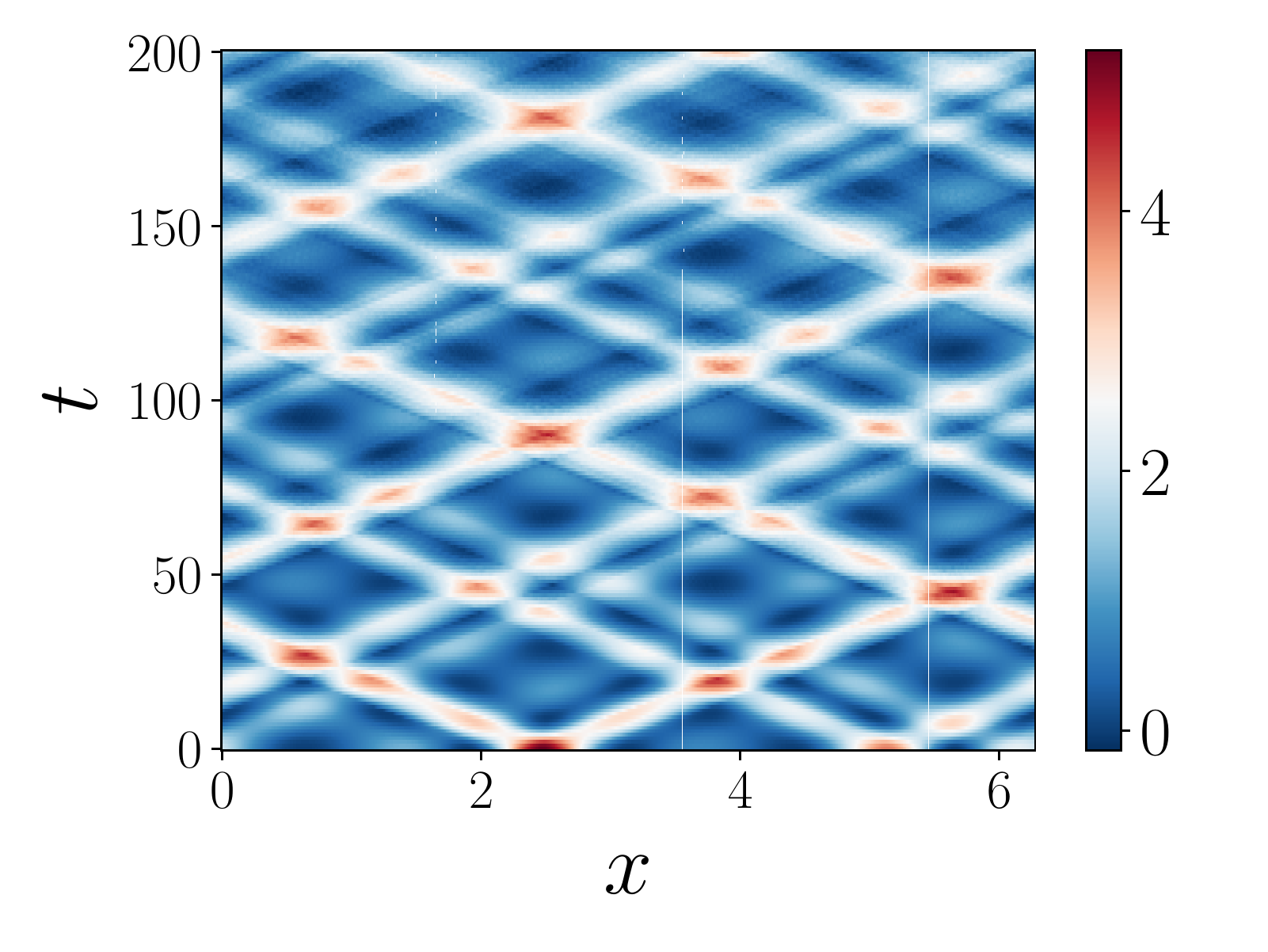}
}
\centerline{
\includegraphics[height=5.0cm]{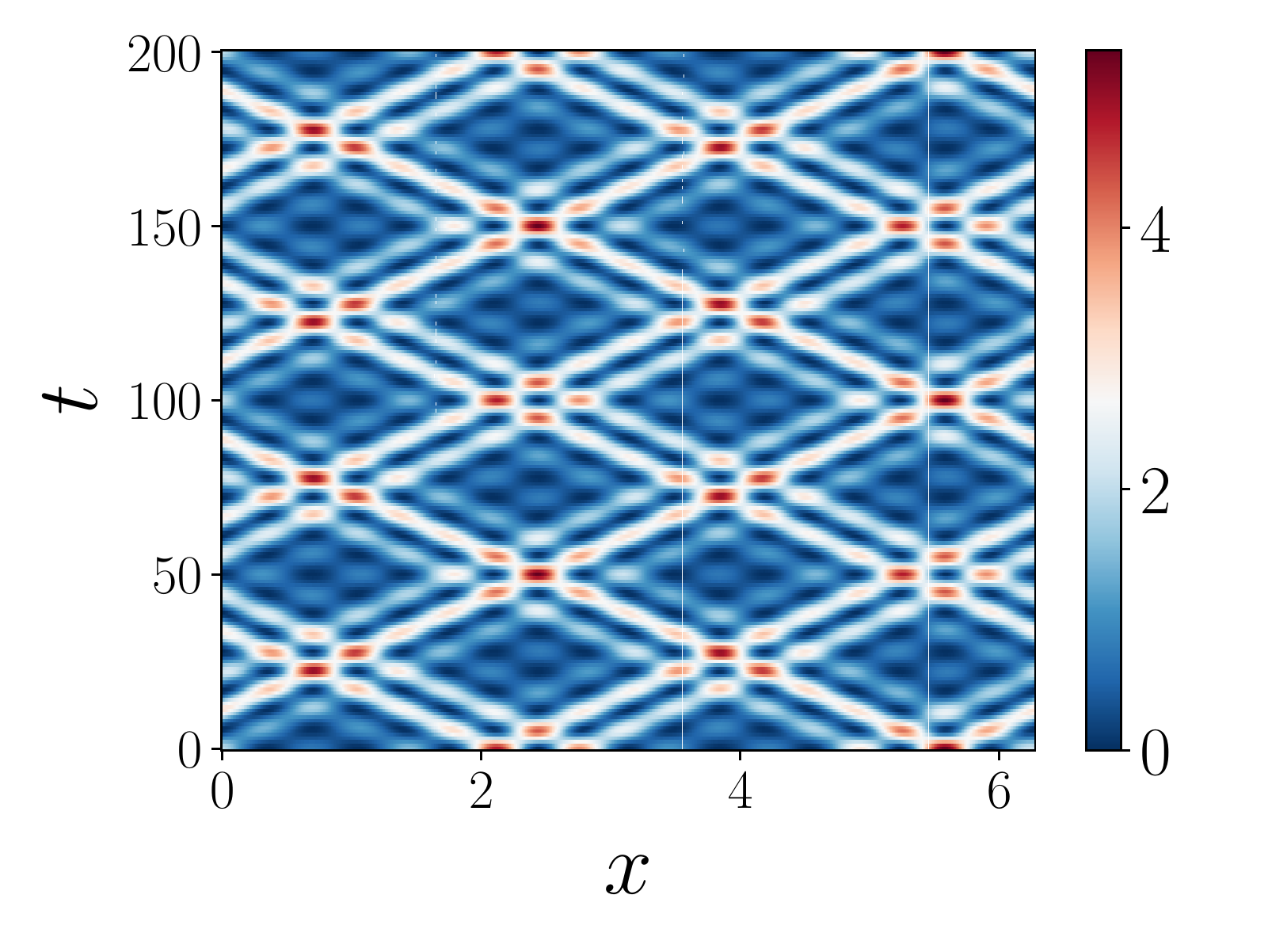}
\includegraphics[height=5.0cm]{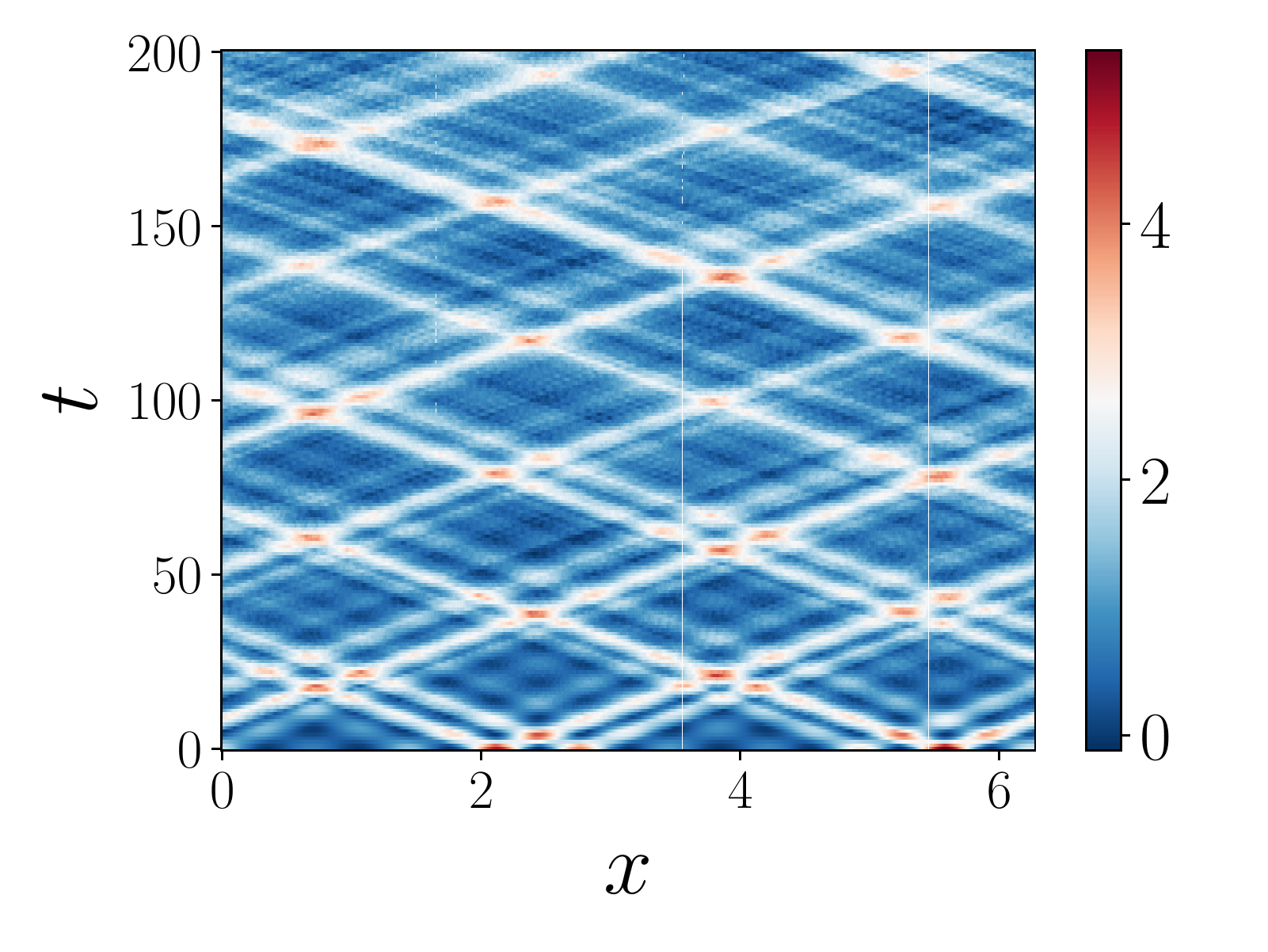}
}
 
\caption{Sample solutions of the nonlinear 
wave equation \eqref{eqn:NLWE} 
with initial conditions $u(x,0)=e^{-\sin(2x)}(1+\cos(x))$ (first row), 
$u(x,0)=e^{-\sin(2x)}(1+\cos(5x))$(second row), and 
$u(x,0)=e^{-\sin(2x)}(1+\cos(9x))$ (third row). 
We set the group velocity $\alpha$ to 
$(2\pi/100)^2$ and consider different nonlinear 
interaction terms: $G=0$ (first column -- linear waves),  
$G=\beta u_x^4/4$ with $\beta=(2\pi/100)^4$ 
(second column -- nonlinear waves). It is seen that as the initial 
condition becomes rougher, the nonlinear effects 
become more important. }
\label{fig:NLW_flow}
\end{figure}

\subsection{Linear waves}
\label{sec:linear_waves}
Setting the interaction term $G(p, u_x,u)$ 
in \eqref{Hamiltonian_NLW} and \eqref{eqn:NLWE}
equal to zero yields the well-known linear wave equation
\begin{align}
\label{eqn:linear_wave}
u_{tt}=\alpha u_{xx}.
\end{align}
We discretize \eqref{eqn:linear_wave} in space 
using second-order finite differences on the (periodic) grid 
$x_j=2\pi j/N$ ($j=0,\dots, N$).
This yields the following linear dynamical 
system
\begin{equation}
\frac{d u_j}{dt}=p_j, \qquad \frac{d p_j}{dt}=\frac{\alpha}{h^2}(u_{j+1}-2u_j+u_{j-1}),
\label{linwave}
\end{equation}
where $u_j(t)=u(x_j,t)$, $p_j(t)=p(x_j,t)$, 
and $h=2\pi/N$ is the  mesh size.
The Hamilton's function corresponding to the 
finite-difference scheme \eqref{linwave} is obtained by 
discretizing the integral \eqref{Hamiltonian_NLW}, e.g., with 
the rectangle rule. This yields  
\begin{align}\label{lattice_H1}
\H_1(\bm p,\bm u)=\sum_{j=0}^{N-1}\frac{h}{2}p_j^2+
\frac{\alpha_1h }{2}\sum_{j=0}^{N-1}(u_{j+1}-u_j)^2,
\end{align}
where we defined $\alpha_1=\alpha/h^2$. 
The corresponding finite-dimensional 
Gibbs distribution can be written as  
\begin{align}\label{Gibbs_1}
\rho_{eq}(\bm p,\bm u) = 
\frac{1}{Z_1(\alpha_1,\gamma)}
\exp\left\{-\gamma\left(\sum_{j=0}^{N-1}\frac{1}{2}p_j^2+
\frac{\alpha_1}{2}\sum_{j=0}^{N-1}(u_{j+1}-u_j)^2\right)\right\},
\end{align}
$Z_1(\alpha_1,\gamma)$ being the partition function 
(normalization constant). Note that we absorbed the 
scaling factor $h$ in the parameter $\gamma > 0$. 
It is straightforward to verify that the lattice 
Hamiltonian \eqref{lattice_H1} is 
preserved if $u_0=u_N$ and $p_0=p_{N}$ 
(periodic boundary conditions). 
This implies that the PDF \eqref{Gibbs_1} 
is invariant under the flow generated by the linear 
ODE \eqref{linwave}. 
Note that the lattice Hamiltonian \eqref{lattice_H1} 
coincides with the Hamiltonian of a one-dimensional 
chain of harmonic oscillators with uniform mass $m=1$ 
and spring constants $k=\alpha_1$. We set $N=100$ 
and $\alpha=(2\pi/100)^2$ in equation \eqref{linwave}. In this 
way, the system \eqref{linwave} is $200$-dimensional and 
the modeling parameter $\alpha_1$ in \eqref{lattice_H1}-\eqref{Gibbs_1} 
is equal to $1$. 

\begin{figure}[t!]
\centerline{\hspace{0.4cm}
${G=0}$\hspace{5.8cm}
${G=\beta u_x^4/4}$
}
\centerline{
\includegraphics[height=5.5cm]{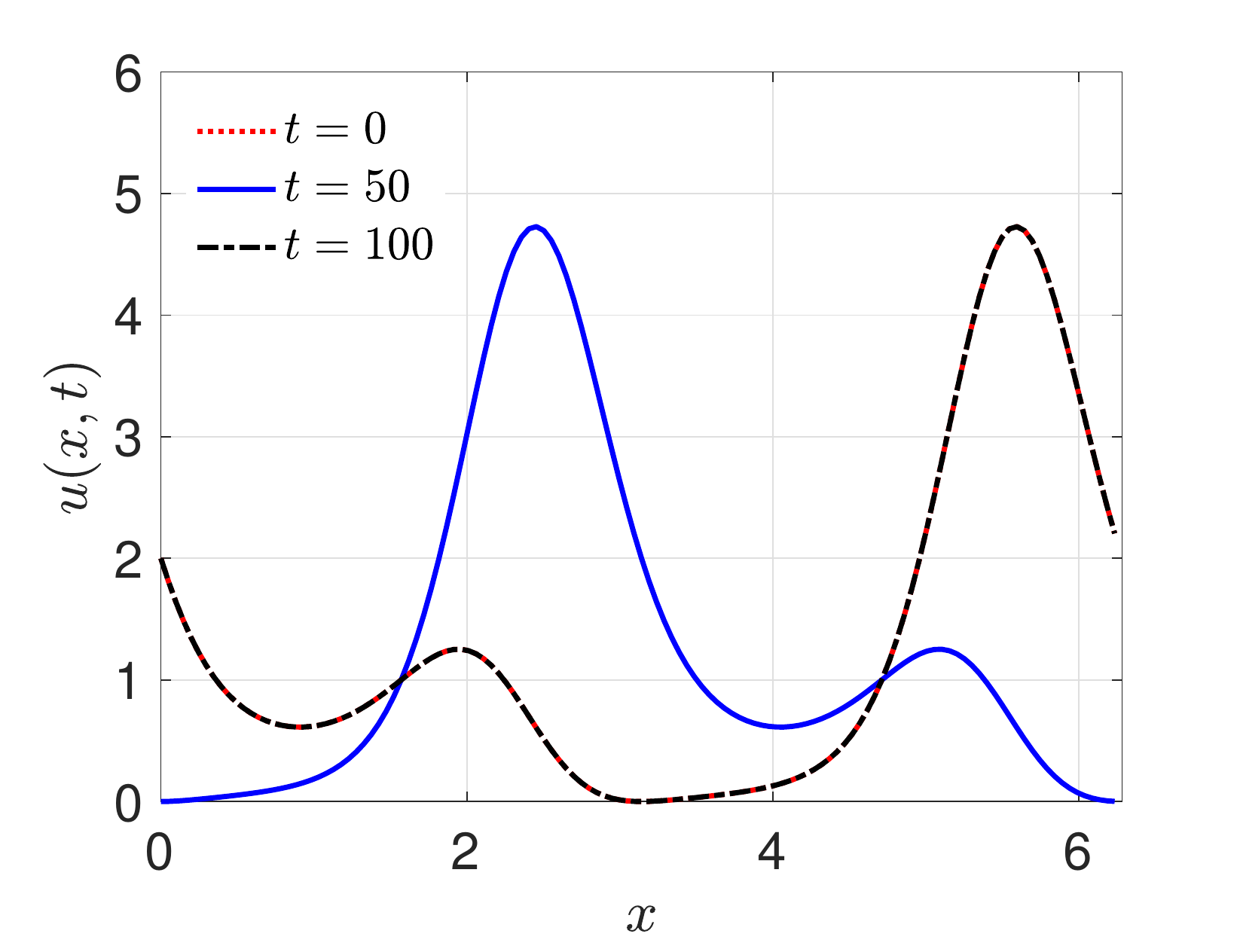} 
\includegraphics[height=5.5cm]{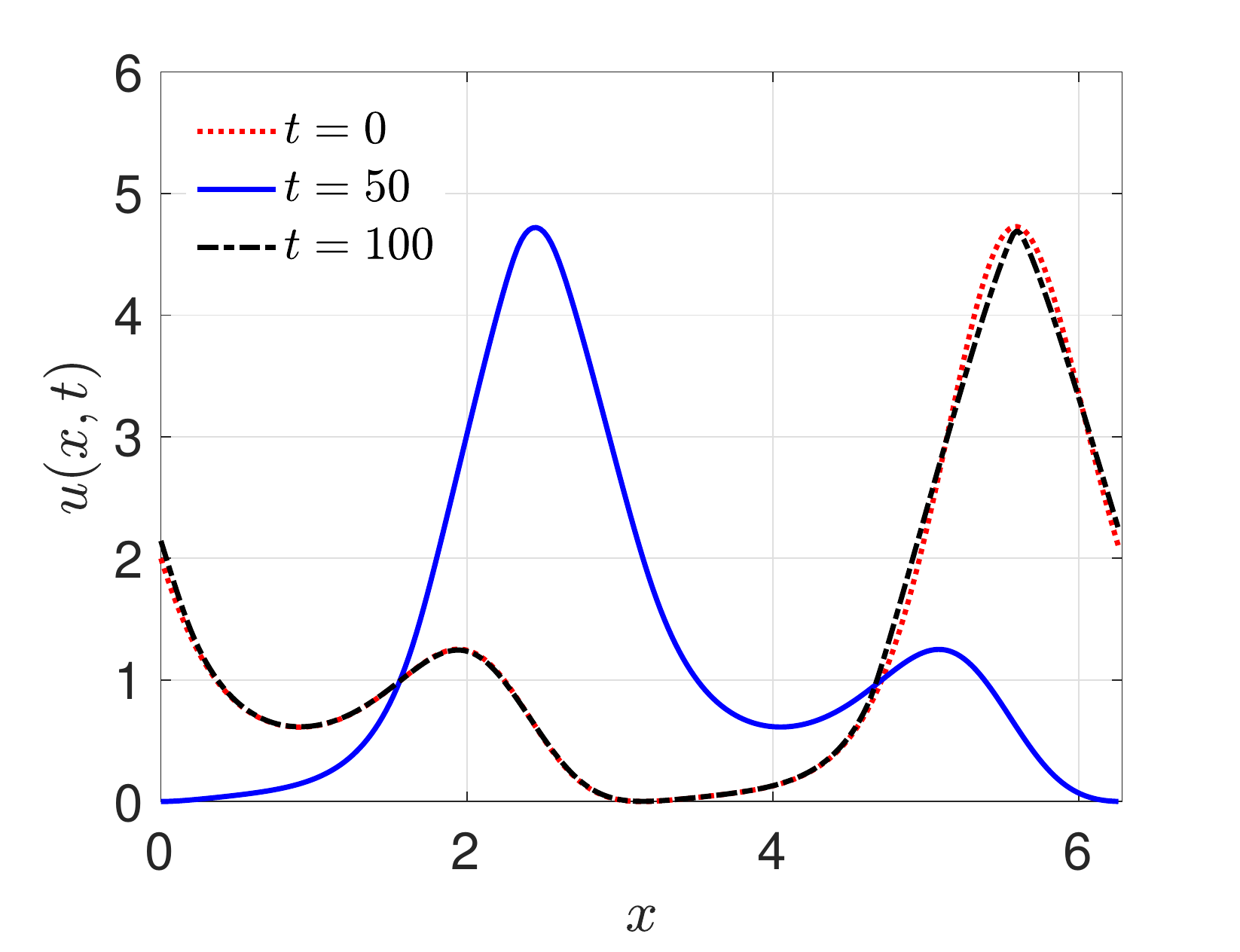}
}
\centerline{
\includegraphics[height=5.5cm]{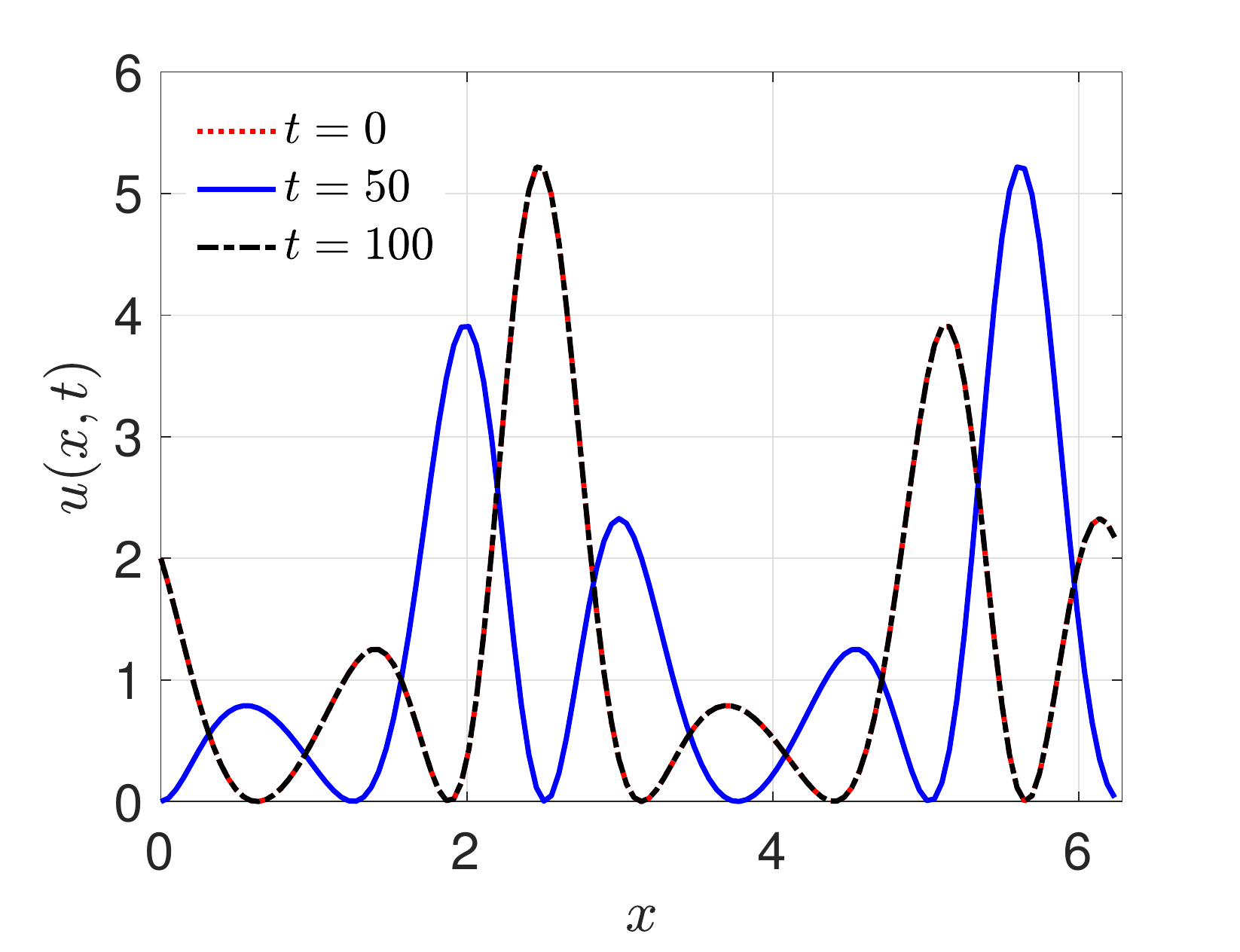}
\includegraphics[height=5.5cm]{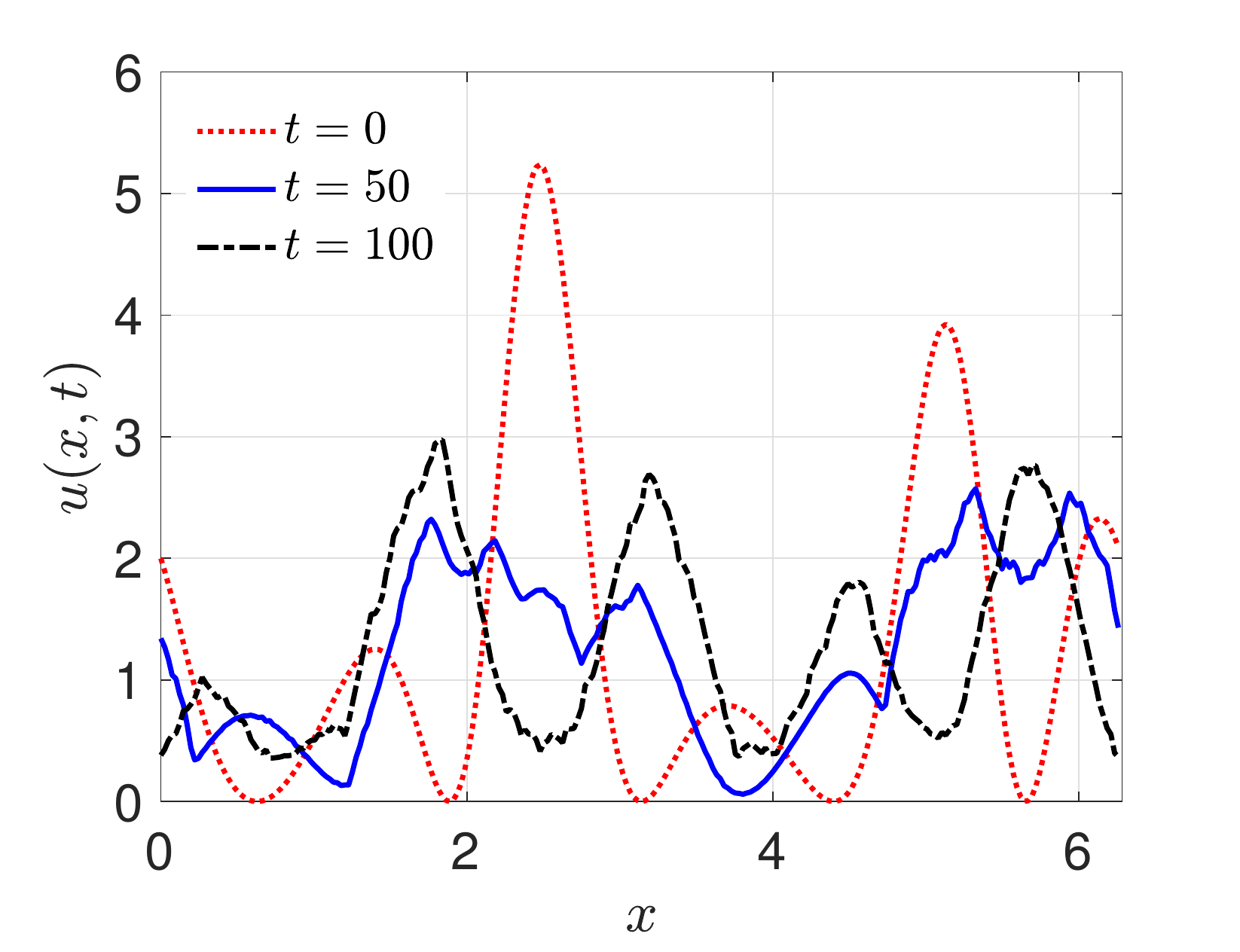}
}
\centerline{
\includegraphics[height=5.5cm]{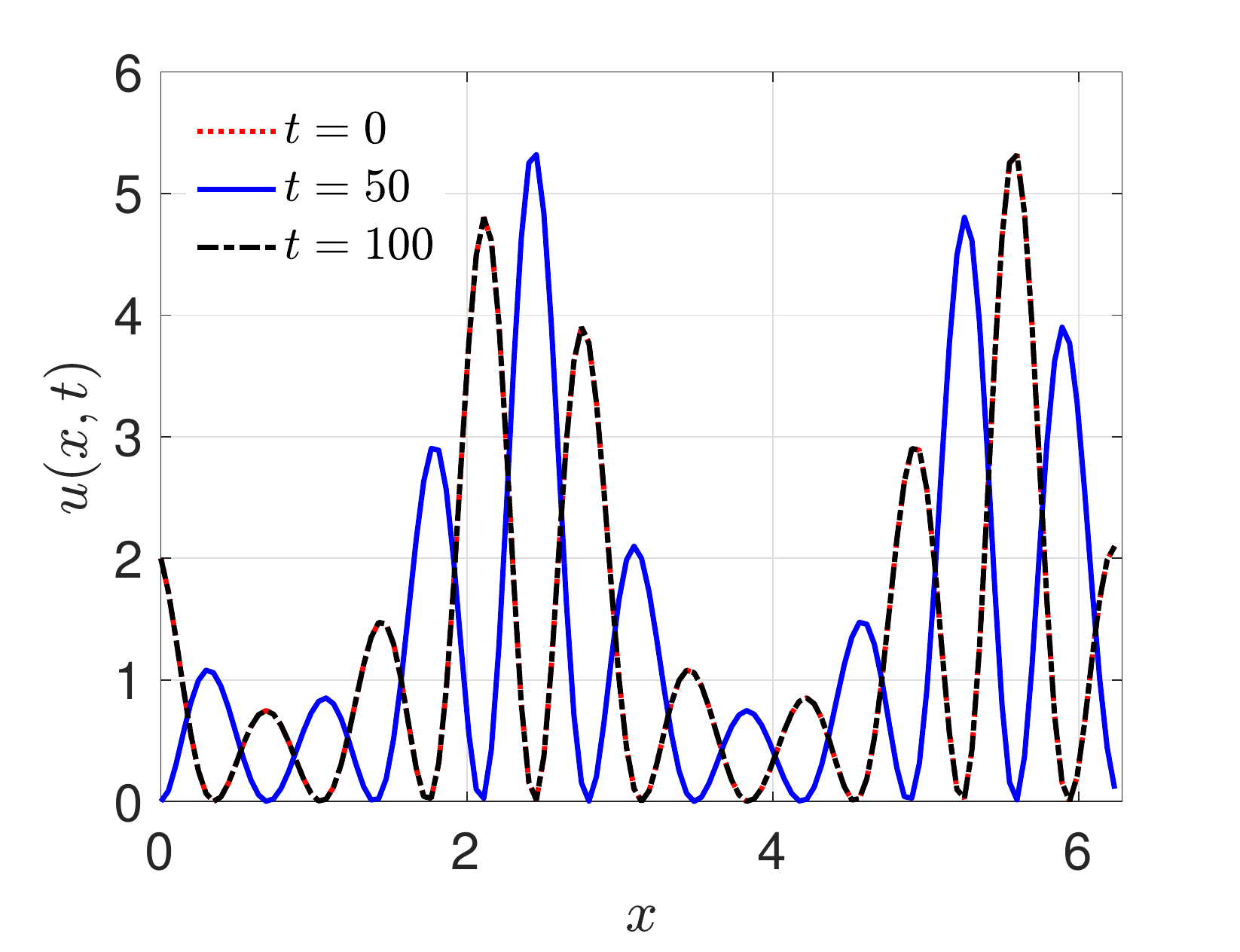}
\includegraphics[height=5.5cm]{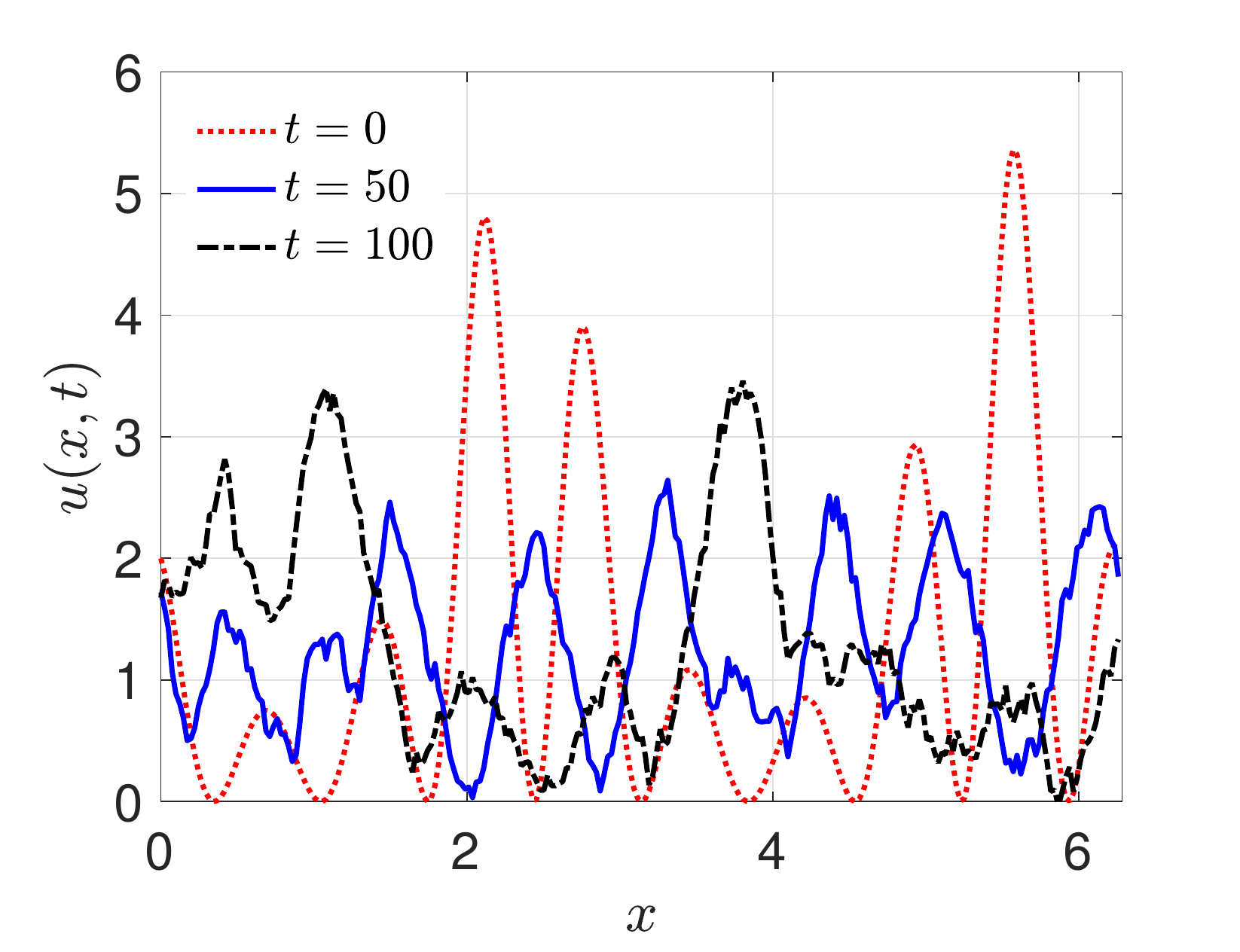}
} 
\caption{Snapshots of the solution shown in Figure \ref{fig:NLW_flow}.}
\label{fig:NLW_path}
\end{figure}

\paragraph{MZ memory kernel and auto-correlation functions} 

The Hamiltonian system \eqref{linwave} with 
periodic boundary conditions has many symmetries. 
In particular, the statistical properties of wave displacement 
$u(x,t)$ at any point $x_j$ are the same, if the 
initial state is distributed according to \eqref{Gibbs_1}. 
In addition, the PDF of the 
wave momentum\footnote{Note that for linear waves 
the wave momentum $p(x,t)$ is equal to 
$\partial u(x,t)/\partial t$ (see Eq. \eqref{eq:wavesystem}).} 
$p(x_j,t)$ and the wave displacement $r(x_j,t)=u(x_{j+1},t)-u(x_j,t)$ 
are both Gaussian (see Eq. \eqref{Gibbs_1}).
Suppose we are interested in the 
temporal auto-correlation function of the  wave 
momentum $p(x_j,t)=p_j$, at an arbitrary location 
$x_j$, i.e., 
\begin{equation}
C_{p_j}(t) = \langle p_j(t),p_j(0)\rangle_{eq},
\label{autoC}
\end{equation}
where $\langle ,\rangle_{eq}$ is an integral over 
the equilibrium distribution \eqref{Gibbs_1}.
Such correlation function admits the analytical expression (see \cite{florencio1985exact}) 
\begin{align}\label{ana}
C_{p_j}(t)= J_0(2t), \qquad \forall \gamma>0,
\end{align}
where $J_0$ is the zero-order Bessel function 
of the first kind. With  $C_{p_j}(t)$ available, we 
can solve the MZ equation 
\begin{align}
\label{gleCC}
\frac{d}{dt}C_{p_j}(t) =  \int_{0}^{t} K(t-s) C_{p_j}(s)ds
\end{align}
for the memory kernel $K(t)$ by using Laplace transforms. 
This yields the exact MZ kernel
\begin{equation}
K(t) = \frac{J_1(2t)}{t},\qquad \forall \gamma>0,
\label{exactK}
\end{equation}
where $J_1$ is the first-order Bessel function of the 
first kind. 
In Figure \ref{fig:correlation_k}, we compare the 
exact memory kernel \eqref{exactK} and the correlation function 
\eqref{gleCC} with the results we obtained using the 
iterative algorithm discussed in Section \ref{sec:iterativealgorithm}. 
Note that the system \eqref{linwave} is linear. 
Therefore, we can use the formula 
\eqref{linear_gamman} to compute the coefficients 
$\{\gamma_1,\dots,\gamma_{n+2}\}$. 
With such coefficients available, 
we then compute $\{\mu_1,\dots,\mu_{n+2}\}$ 
using the recurrence relation \eqref{iterative}, 
and the MZ memory kernel \eqref{Mq}. 
 {
In Figure \eqref{fig:correlation_k} we demonstrate 
that the MZ-Faber expansion rapidly converges to 
the exact auto-correlation function \eqref{autoC} of the 
wave momentum as we increase the  Faber 
expansion order $n$. This is not surprising since 
the linear wave equation is a well-known 
integrable system for which convergence of 
the MZ-Faber series can be rigorously 
established (\S 5 in \cite{zhu2018faber}).}
\begin{figure}
\centerline{
\includegraphics[height=5.0cm]{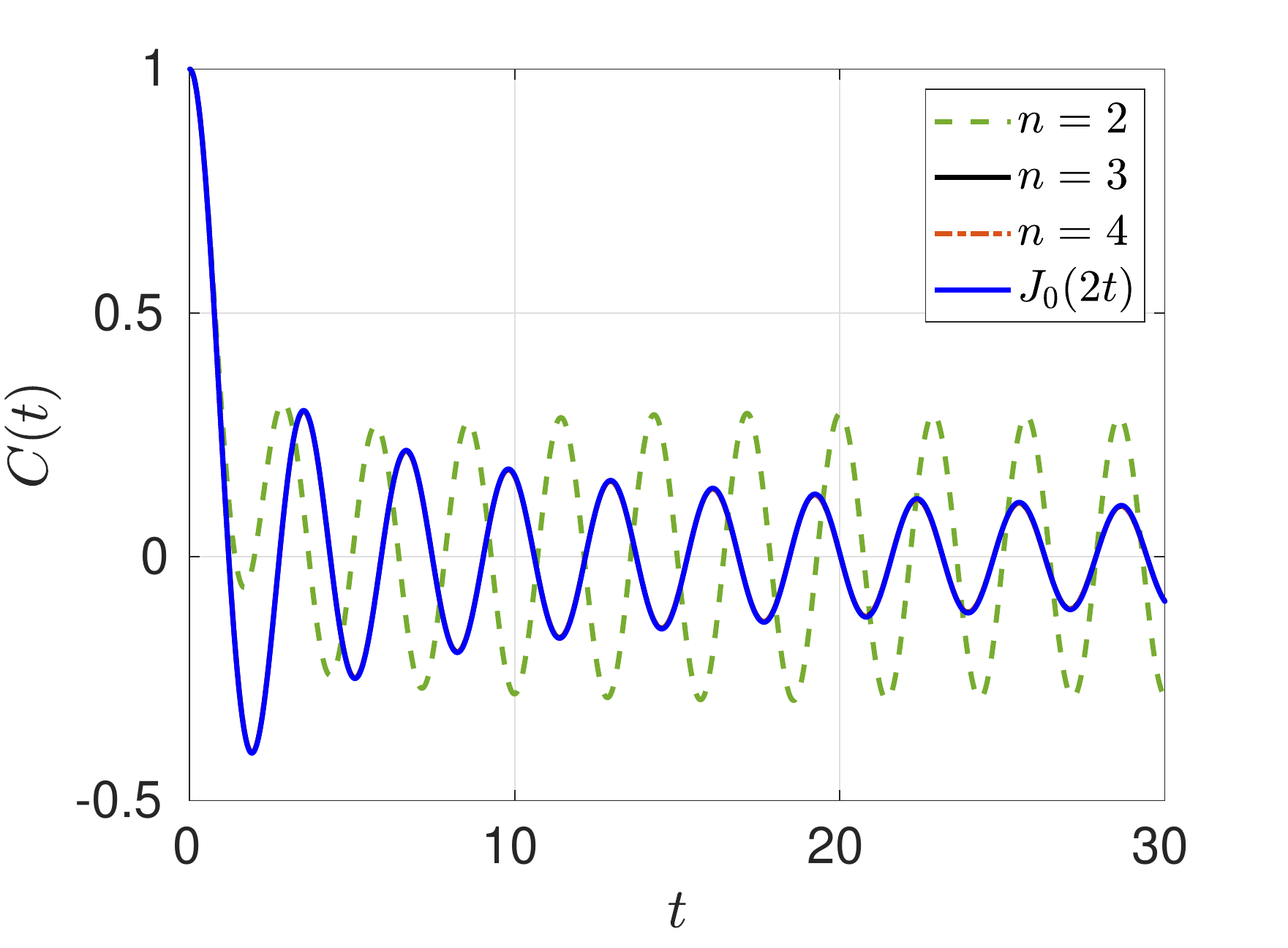} 
\includegraphics[height=5.0cm]{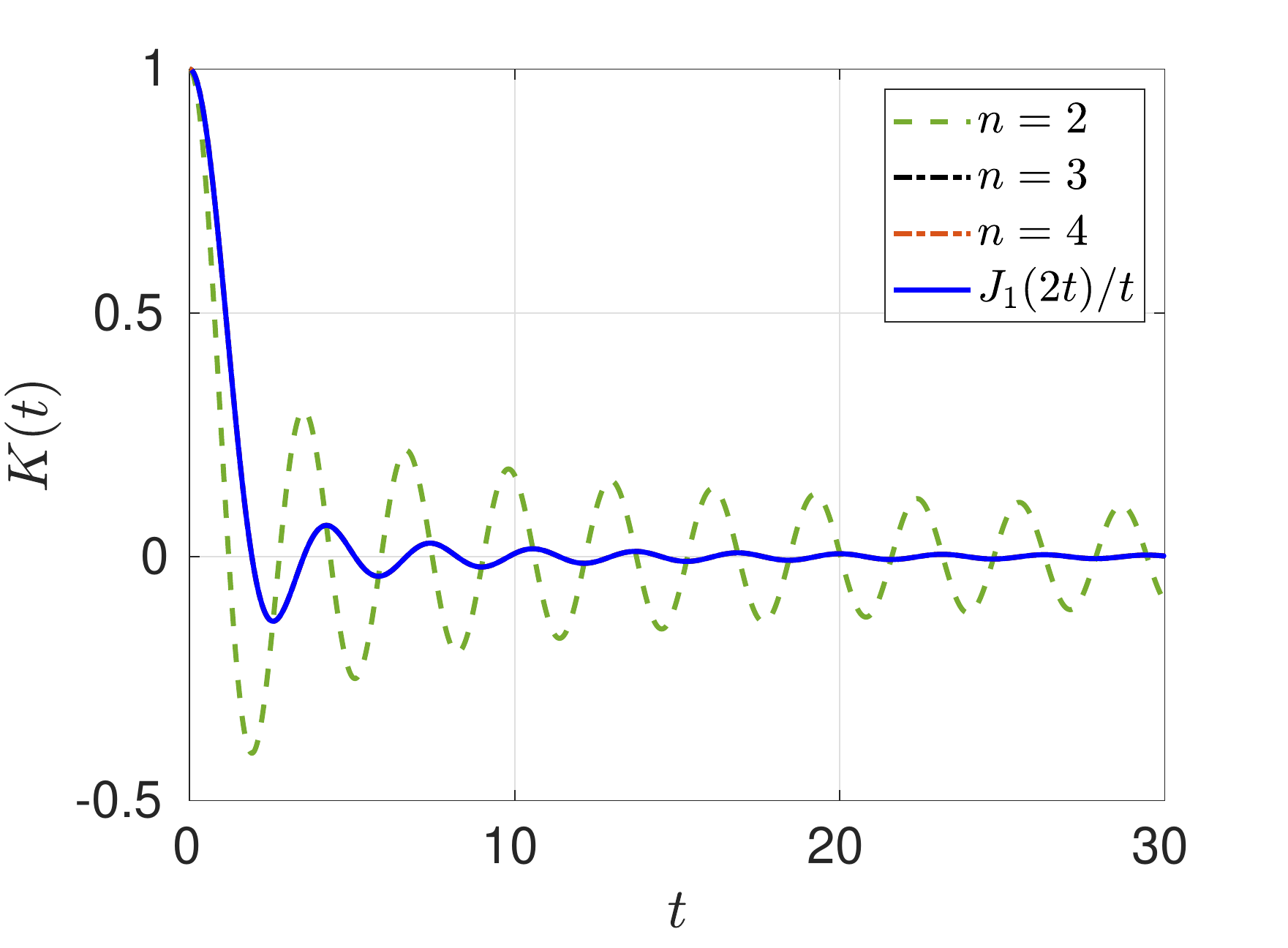}
}
\caption{{  Linear wave equation \eqref{eqn:linear_wave}. 
Temporal auto-correlation function  of the 
wave momentum $p(x_j,t)={\partial u(x_j,t)}/\partial t$ 
(Eq. \eqref{autoC}, any location $x_j$) and MZ memory 
kernel $K(t)$. We compare the the analytical 
results \eqref{ana} and \eqref{exactK}, with results 
we obtained by using the recursive 
algorithm we presented in Section \ref{sec:recursive} 
for different Faber polynomial orders $n$. It is seen that 
the MZ-Faber expansion rapidly converges to 
the exact MZ-kernel and auto-correlation function we 
increase the polynomial order.}}
\label{fig:correlation_k} 
\end{figure}

\paragraph{Reduced-order stochastic modeling}
Suppose we are interested in building a consistent 
reduced-order stochastic model for the wave momentum 
$p(x_j,t)=\partial u(x_j,t)/\partial t$ at 
statistical equilibrium. To this end, we employ the spectral 
expansion technique we discussed in Section \ref{sec:Model}. 
The auto-correlation function of the process 
$p(t)=p(x_j,t)$ (at any location $x_j$), i.e., 
\eqref{autoC}, is obtained by solving the 
MZ equation \eqref{gleCC} with the kernel 
computed using the combinatorial algorithm described 
in Section \ref{sec:iterativealgorithm}.
 {Following the stochastic modeling paradigm 
we developed in Section \ref{sec:Model}}, we expand $p(t)$ as 
\begin{equation}
\label{KLP}
p(t)\simeq \sum_{k=1}^{K}\sqrt{\lambda_k}\xi_k(\omega)e_k(t),
\end{equation}
where $(\lambda_k,e_k(t))$ are eigenvalues and eigenfunctions 
of \eqref{autoC}. 
By enforcing consistency of \eqref{KLP} with 
the equilibrium distribution \eqref{Gibbs_1} at each fixed 
time we obtain that the random variables $p(t_j)$ 
are normally distributed with zero mean and 
variance $1/\gamma$, for all $t_j\in [0,10]$. 
In other words $p(t)$ is a centered, stationary 
Gaussian random process with 
correlation function \eqref{autoC}. 
In Figure \ref{fig:C_Harmonic}, we plot 
the auto-correlation functions 
\begin{align}\label{time_c}
C_p(t)=\langle p_j(t),p_j(0)\rangle_{eq},\qquad
C_p^2(t)=\langle p^2_j(t),p^2_j(0)\rangle_{eq},\qquad
C_p^4(t)=\langle p^4_j(t),p^4_j(0)\rangle_{eq},
\end{align} 
we obtained with an MZ-Faber expansion of 
degree $n =6$. Convergence of KL expansions representing 
high-order correlation functions such as \eqref{time_c}
is established in Appendix.
\begin{figure}
\centering
\includegraphics[height=3.9cm]{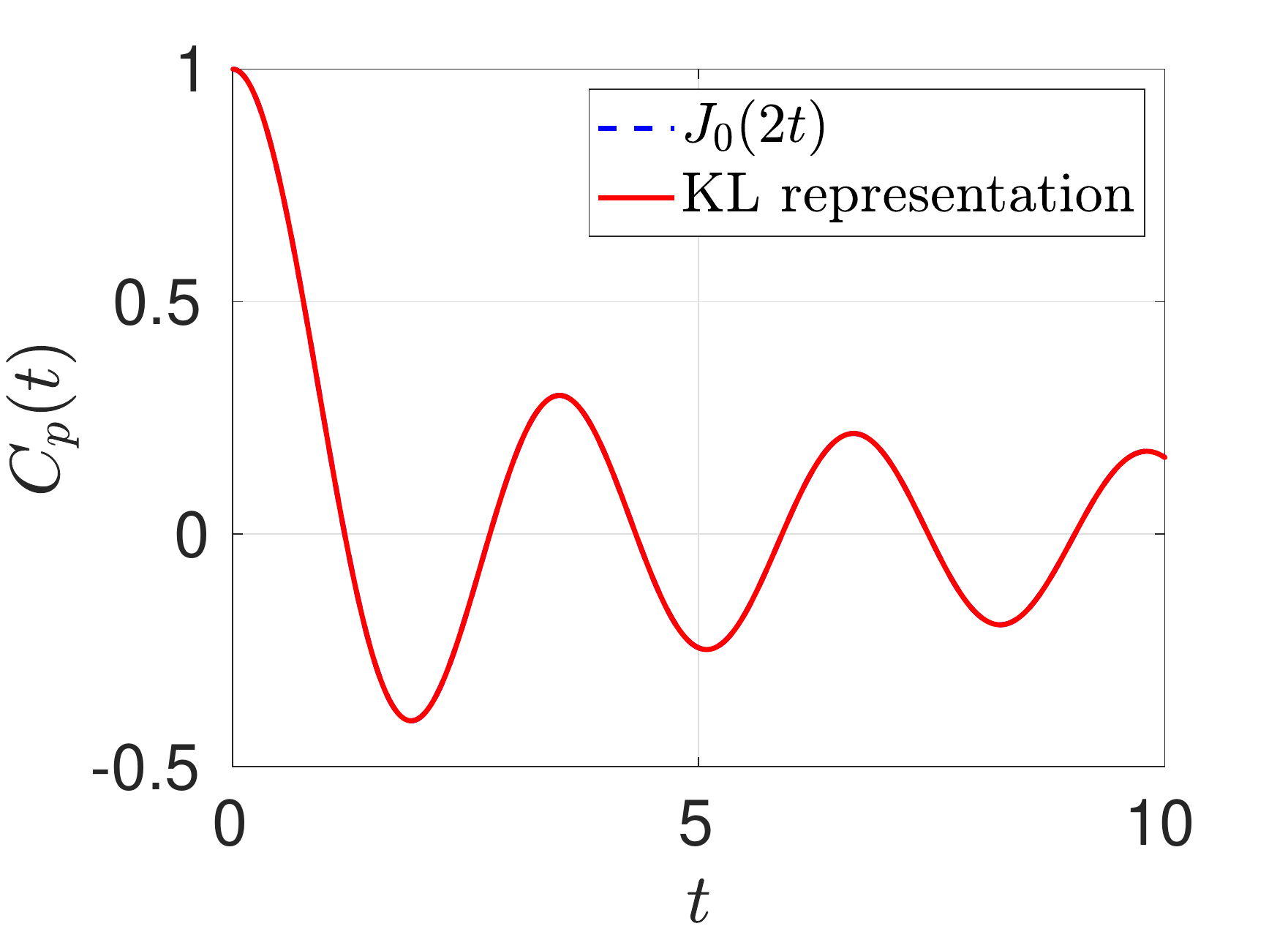} 
\includegraphics[height=3.9cm]{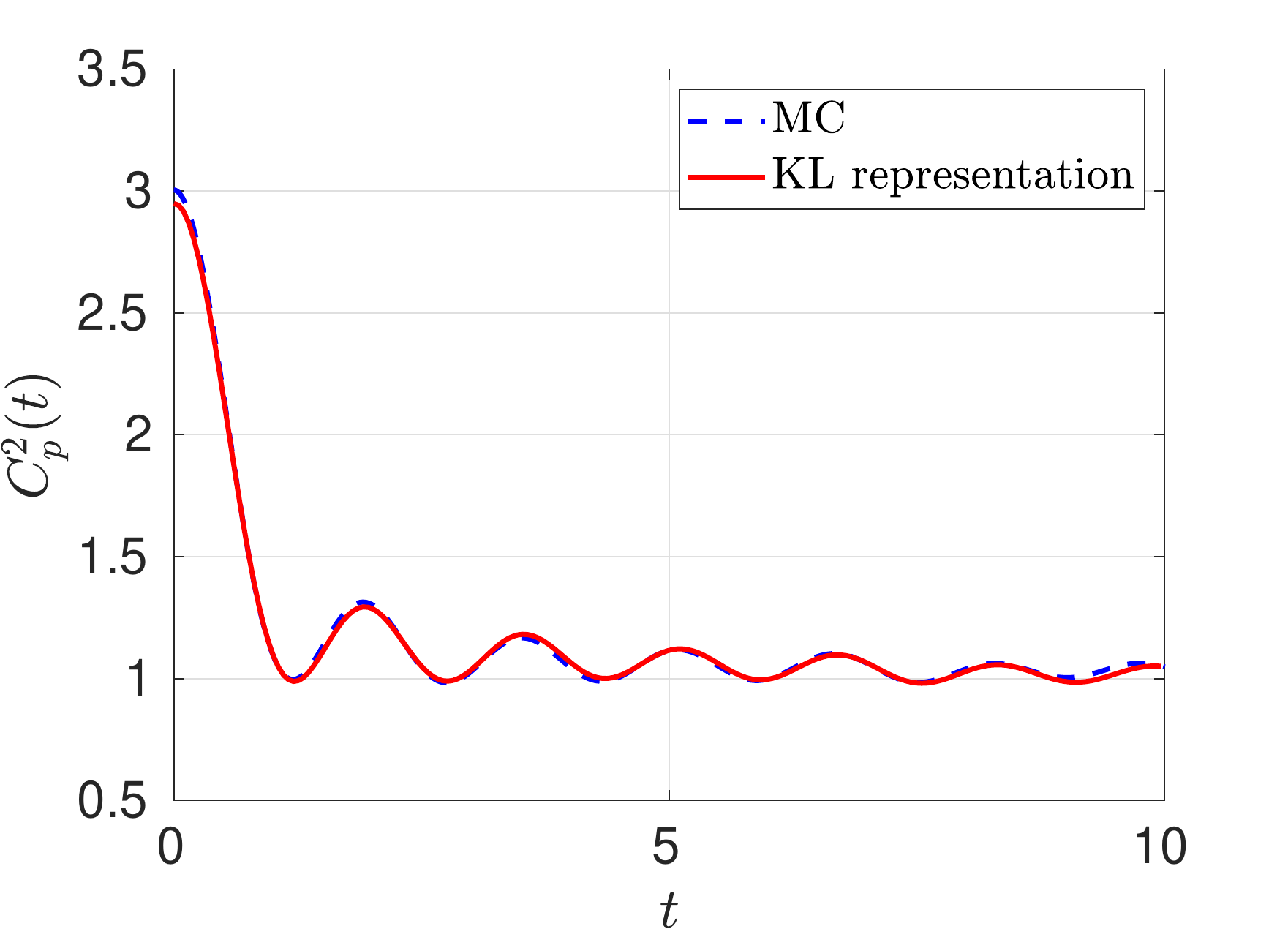}
\includegraphics[height=3.9cm]{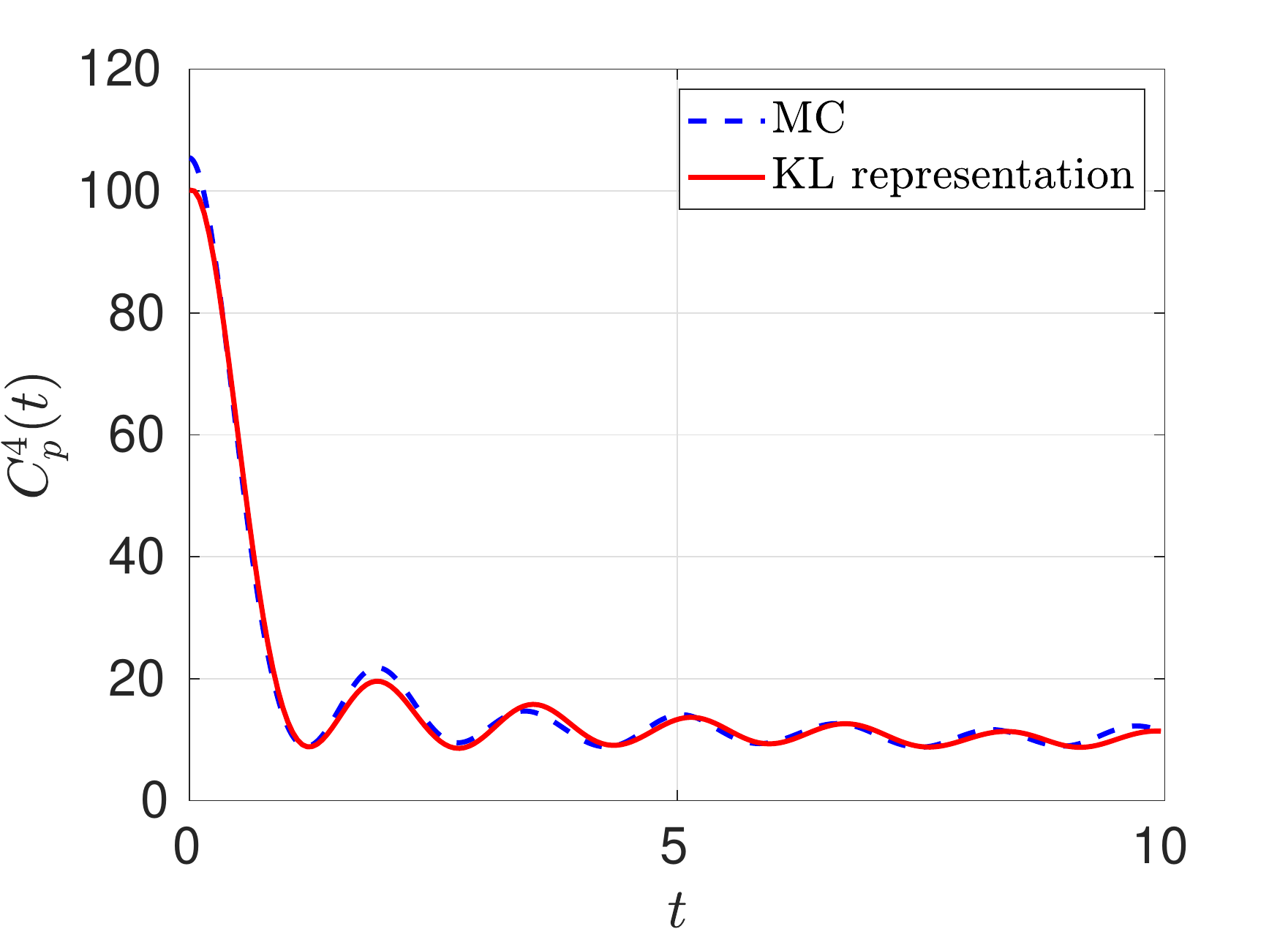}
\caption{Linear wave equation \eqref{eqn:linear_wave}. 
Temporal auto-correlation functions \eqref{time_c} 
of the wave momentum. {  The MZ kernel 
here is approximated  with a Faber polynomial series of 
degree $n=10$.}} 
\label{fig:C_Harmonic}
\end{figure}

\subsection{Nonlinear waves}
\label{sec:FPU}
Here we study the nonlinear wave equation 
\eqref{eqn:NLWE} with interaction term
$G(p,u_x,u)=\beta u_x^4/4$, i.e., 
\begin{align}\label{eqn:nonlinear_wave1}
u_{tt}=\alpha u_{xx}+3\beta u_x^2u_{xx},\qquad \alpha,\beta>0.
\end{align}
In Figure \ref{fig:NLW_flow} and Figure \ref{fig:NLW_path} 
we plot sample solutions of \eqref{eqn:nonlinear_wave1} 
corresponding to different initial conditions. 
It is clearly seen that the nonlinearity 
$u_x^2u_{xx}$ breaks the periodicity of 
traveling wave. This effect is more pronounced  if 
the initial condition is rougher in $x$, as $u^2_x$ and $u_{xx}$
are larger in this case, thereby increasing magnitude of the nonlinear 
term in \eqref{eqn:nonlinear_wave1}. As before, 
we discretize \eqref{eqn:nonlinear_wave1} 
and the Hamiltonian \eqref{Hamiltonian_NLW} with finite 
differences on a periodic spatial grid ($N$ 
points in $[0,2\pi]$). This yields
\begin{align}\label{lattice_H2}
\H_2(\bm p,\bm u)=
\sum_{j=0}^{N-1}\frac{hp_j^2}{2}+\sum_{j=0}^{N-1}
\frac{h\alpha_1}{2}(u_{j+1}-u_j)^2+\sum_{j=0}^{N-1}\frac{h\beta_1}{4}(u_{j+1}-u_j)^4,
\end{align}
where $u_j(t)=u(x_j,t)$ and $p_j(t)=\partial u(x_j,t)/\partial t$ 
represent the wave amplitude and momentum at location $x_j=h j$ ($j=0,\dots, N$,  $h=2\pi/N$), 
$\alpha_1=\alpha/h^2$ and $\beta_1=\beta/h^4$. 
The discretized equilibrium distribution 
\eqref{equlibrium} then becomes
\begin{align}\label{Gibbs_2}
\rho_{eq}(\bm p,\bm u)=\frac{1}{Z_2(\alpha_1,\beta_1,\gamma)}
\exp\left\{-\gamma\left(\sum_{j=0}^{N-1}\frac{p_j^2}{2}+
\sum_{j=0}^{N-1}\frac{\alpha_1}{2}(u_{j+1}-u_j)^2+
\sum_{j=0}^{N-1}\frac{\beta_1}{4}(u_{j+1}-u_j)^4\right)\right\}.
\end{align}
As before,  we absorbed the factor $h$ into the 
parameter $\gamma$. Note that the lattice 
Hamiltonian \eqref{lattice_H2} coincides 
with the Hamiltonian of the Fermi-Pasta-Ulam  
$\beta$-model \eqref{FPU_H}, with $m_j=1$. 
We emphasize that if a different scheme is used to 
discretize the wave equation \eqref{eqn:nonlinear_wave1}, 
then the lattice Hamiltonian \eqref{lattice_H2} 
may not be a conserved quantity.

\paragraph{MZ memory term and auto-correlation functions} 
We choose the wave momentum $p_j(t)$ and 
the wave displacement $r_j(t)=u_{j+1}(t)-u_j(t)$ as quantities 
of interest.  {Moreover, we set $N=100$ and $\alpha=(2\pi/100)^2$. To study the effects of the nonlinear interaction 
term, we consider different values of $\beta=\beta_1\alpha^2$, 
with $\beta_1$ ranging from $0.01$ to $1$. This corresponds 
to the FPU models with mild and strong 
nonlinearities, respectively.
Based on the structure of the Hamiltonian 
\eqref{lattice_H2} and the equilibrium distribution \eqref{Gibbs_2}, 
we expect that the dynamics of $p_j(t)$ and $r_j(t)$ will 
be different for different parameters $\beta$.}
To calculate the temporal auto-correlation 
function of $p_j(t)$ and $r_j(t)$ at an arbitrary 
spatial point $x_j$, we solve the corresponding 
MZ equations. Such equations are of the form \eqref{gleCC}, 
where the memory kernel $K(t-s)$ is computed from 
first-principles (i.e., from the microscopic equations of 
motion) using the algorithm we presented in 
Section \ref{sec:iterativealgorithm}. 
In Figure \ref{fig:FPU_C(t)}, we compare the temporal 
auto-correlation function we obtained for the wave 
displacement $r_j(t)$ with results of 
Markov-Chain-Monte-Carlo (MCMC)  
($10^6$ sample paths)  {for FPU systems 
with mild nonlinearities ($\beta_1=0.01$ and $\beta_1=0.1$ ) 
at different temperatures ($\gamma=1$ and $\gamma=40$).
It is seen that the MZ-Faber approximation of the MZ memory 
kernel yields relatively accurate results for FPU systems 
with mild nonlinearties at both low ($\gamma=40$) 
and high temperature ($\gamma=1$) as we increase 
the polynomial order $n$. }
\textbf{\begin{figure}[ht]
\centerline{\hspace{0.4cm}
${\beta_1=0.01}$\hspace{6cm}
${\beta_1=0.1}$
}
\centerline{
\includegraphics[height=5.5cm]{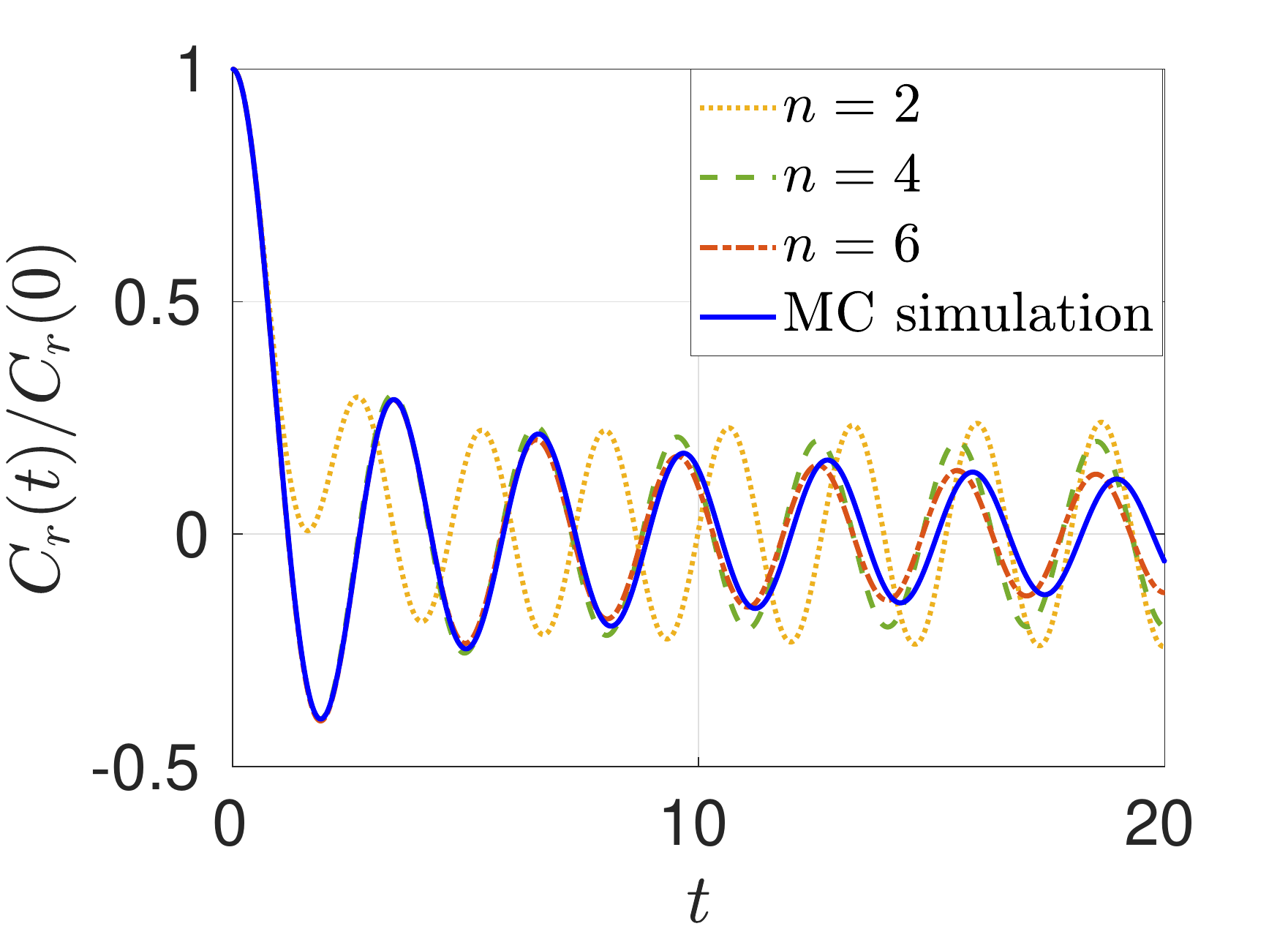} 
\includegraphics[height=5.5cm]{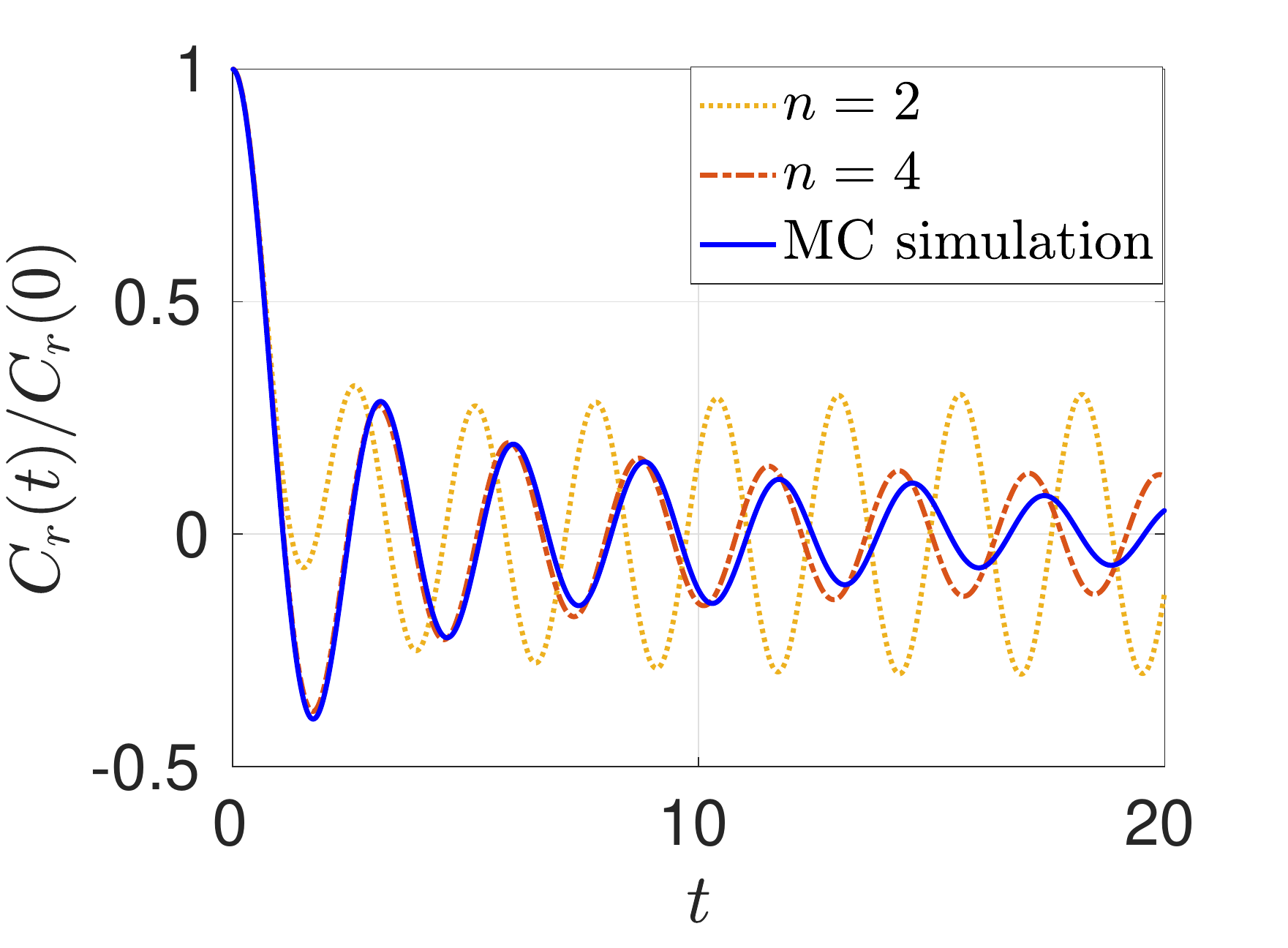}
}
\centerline{
\includegraphics[height=5.5cm]{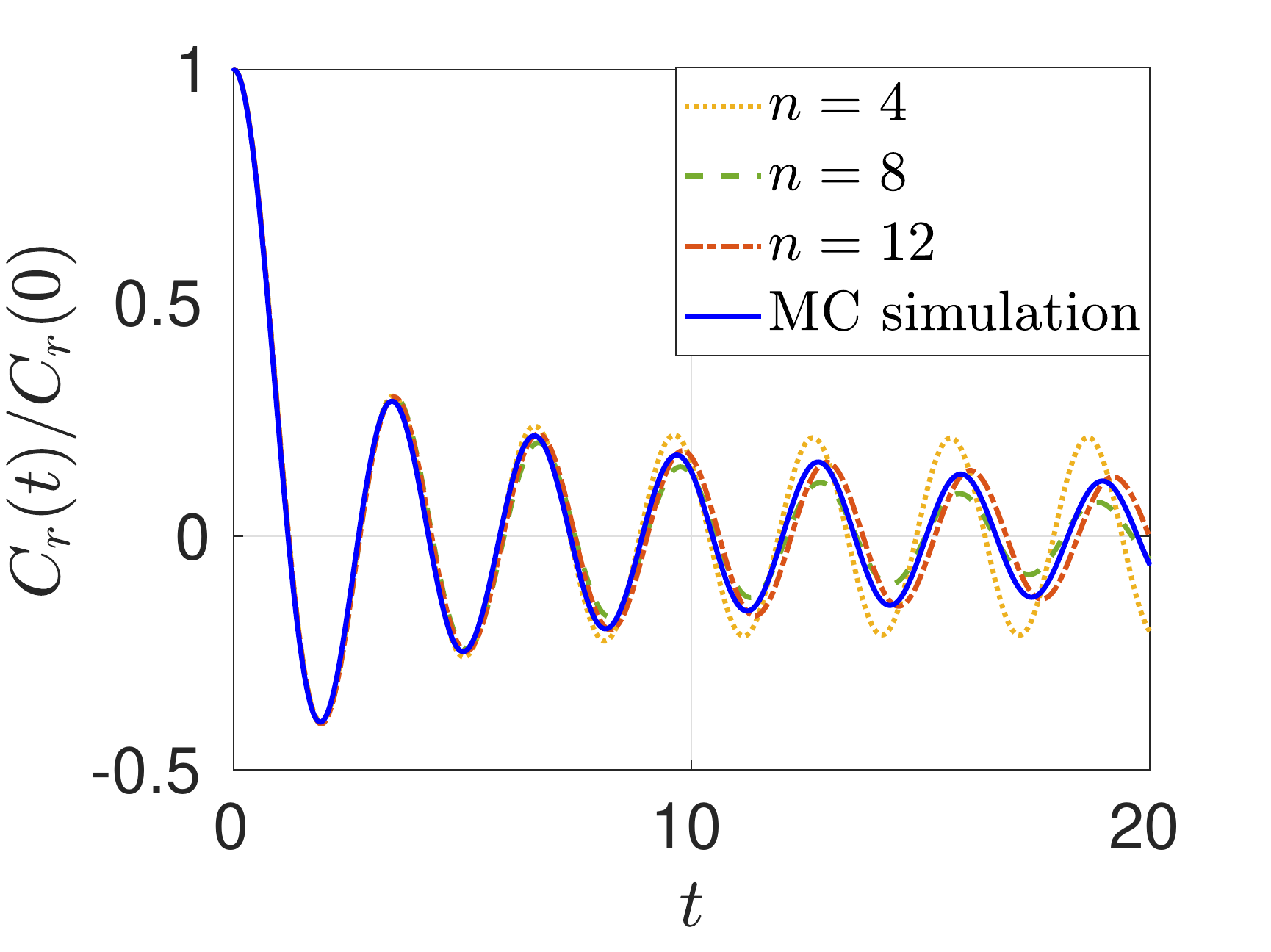}
\includegraphics[height=5.5cm]{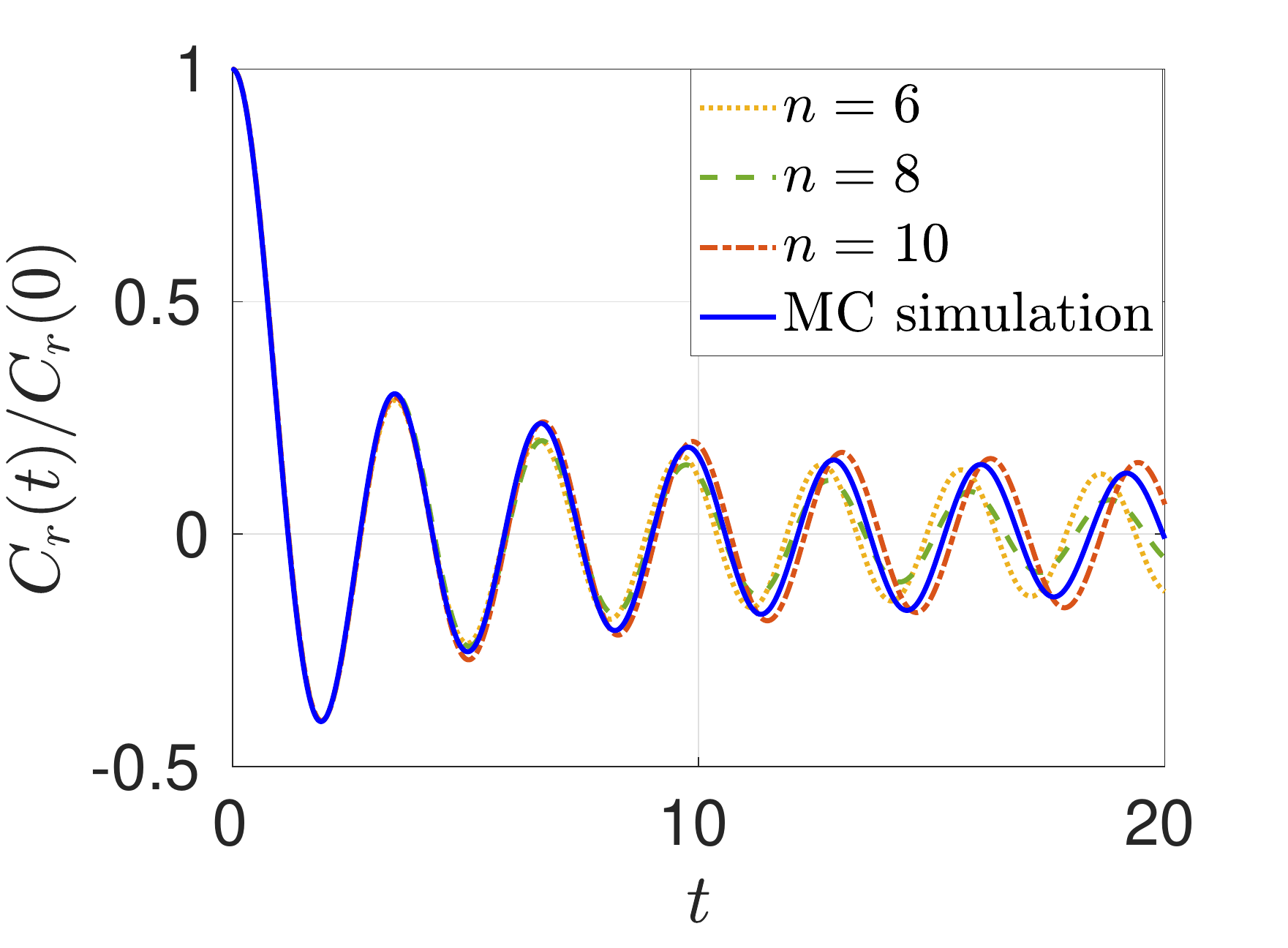}
}
\caption{
 {Nonlinear wave equation 
\eqref{eqn:nonlinear_wave1}. Temporal 
auto-correlation function of the wave displacement $r_j(t)$ 
for different values of the nonlinear parameter 
$\beta_1$. We compare results we obtained by calculating 
the MZ memory from first principles using $n$-th 
order Faber polynomials (Section \ref{sec:calculation_of_gamma}) 
with results from Markov-Chain-Monte-Carlo  
($10^6$ sample paths). The thermodynamic 
parameter $\gamma$ is set to 1 (high-temperature) 
in the first row and to 40 (low-temperature) 
in the second row.}}
\label{fig:FPU_C(t)}
\end{figure}
}

\paragraph{Reduced-order stochastic modeling} 
We employ the spectral approach 
of Section \ref{sec:Model} to 
build stochastic low-dimensional 
models of the wave momentum $p_j(t)$ and 
wave displacement $r_j(t)=u_{j+1}(t)-u_j(t)$ at 
statistical equilibrium.
Since we assumed that we are at statistical 
equilibrium, the statistical properties of the random 
processes representing $p_j(t)$ and $r_j(t)$ 
are time-independent. For instance, by integrating 
\eqref{Gibbs_2} we obtain the following expression for 
the one-time PDF of $r_j(t)$ 
\begin{equation}
r_j(t) \sim  e^{-\gamma(\frac{1}{2}\alpha_1 r^2+\frac{1}{4}\beta_1 r^4)}
\qquad \forall t\in[0,T],  \quad \forall j=0,\dots, N-1.
\label{PFD_r}
\end{equation}
Clearly, $r_j(t)$ is a stationary non-Gaussian process. 
To sample the KL expansion of $r_j(t)$ in a way 
that is consistent with the PDF 
\eqref{PFD_r} we used the algorithm discussed in 
\cite{phoon2002simulation,phoon2005simulation}.
 {For the FPU system with $\alpha_1=\beta_1=1$, it is straightforward to show that 
for all $m\in \mathbb{N}$}
\begin{align*}
\mathbb{E}\{r_j^{2m}(t)\}=\frac{\displaystyle\int_{-\infty}^{+\infty}r^{2m}e^{-\gamma(\frac{1}{2}r^2-\frac{1}{4}r^4)}dr}
{\displaystyle\int_{-\infty}^{+\infty}e^{-\gamma(\frac{1}{2}r^2-\frac{1}{4}r^4)}dr}
=\frac{\sqrt{2} \gamma^{-\frac{1}{4}-\frac{m}{2}}\Gamma\left(\frac{1}{2}+m\right)U\left(\frac{1}{4}+\frac{m}{2},\frac{1}{2},\frac{\gamma}{4}\right)}
{e^{\gamma/8} K_{1/4}\left(\frac{\gamma}{8}\right)},
\end{align*}
where $\Gamma(x)$ is the Gamma function, $K_n(z)$ is the modified Bessel 
function of the second kind and $U(x,y,z)$ is Tricomi's confluent 
hypergeometric function. Therefore, for all positive $\gamma$ 
and finite $m$ we have that 
$\mathbb{E}\{r_j^{2m}(t)\}<\infty$, i.e., $r_j(t)$ 
is $L^{2m}$ process. This condition guarantees 
convergence of the KL expansion to temporal correlation 
functions of order greater than two (see Appendix A).
In Figure \ref{fig:C_FPU} we plot the temporal 
auto-correlation function of various 
polynomial observables of the nonlinear 
wave momentum and displacement at an 
arbitrary spatial point $x_j$. We 
compare results we obtained from Markov Chain Monte Carlo 
simulation (dashed line), with the MZ-KL expansion method 
based the first-principle memory calculation (continuous line). 
We also provide results we obtained by using KL expansions 
with covariance kernel estimated from data (dotted line).

\begin{figure}[t]
\centering

\centerline{
\includegraphics[height=4.cm]{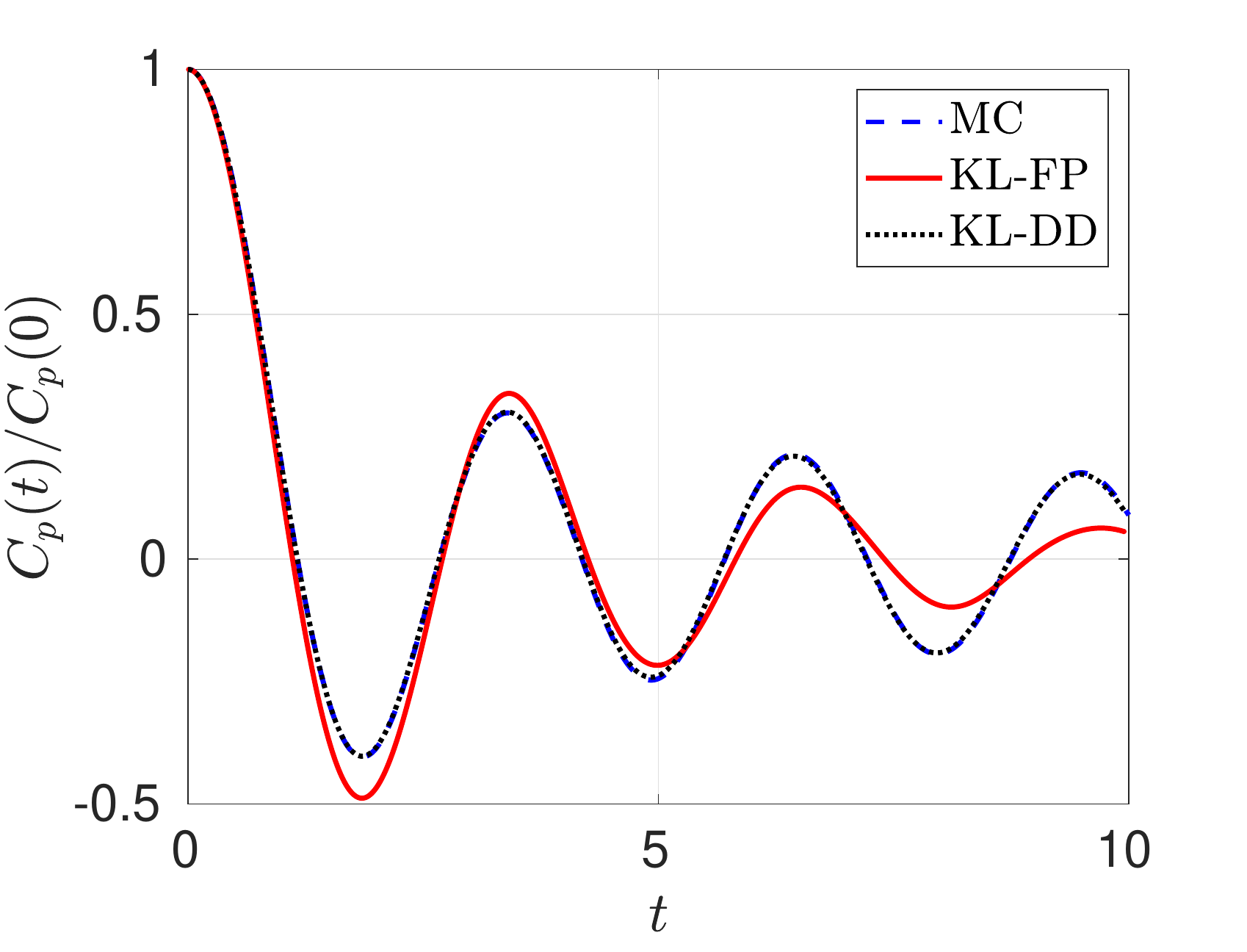} 
\includegraphics[height=4.cm]{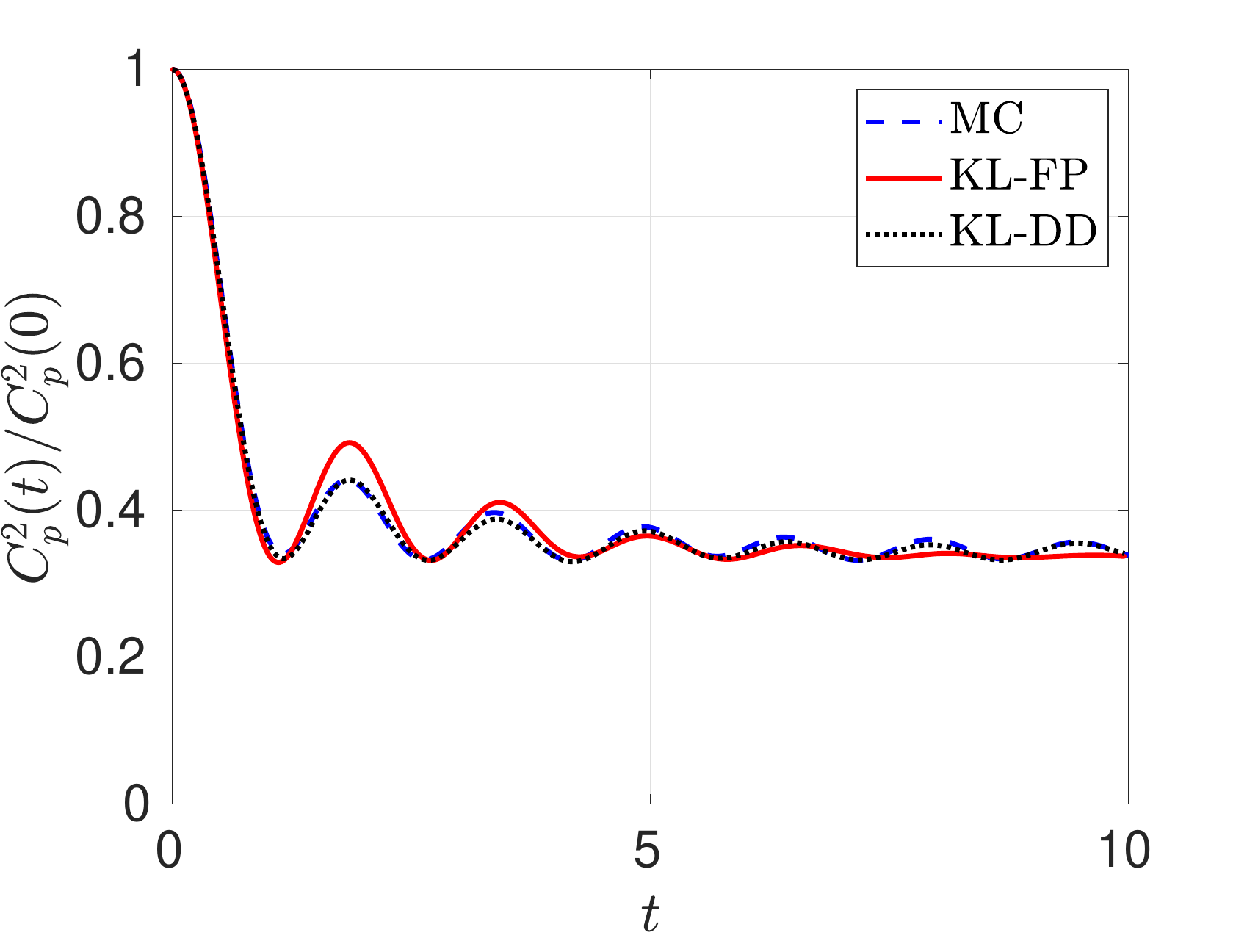}
\includegraphics[height=4.cm]{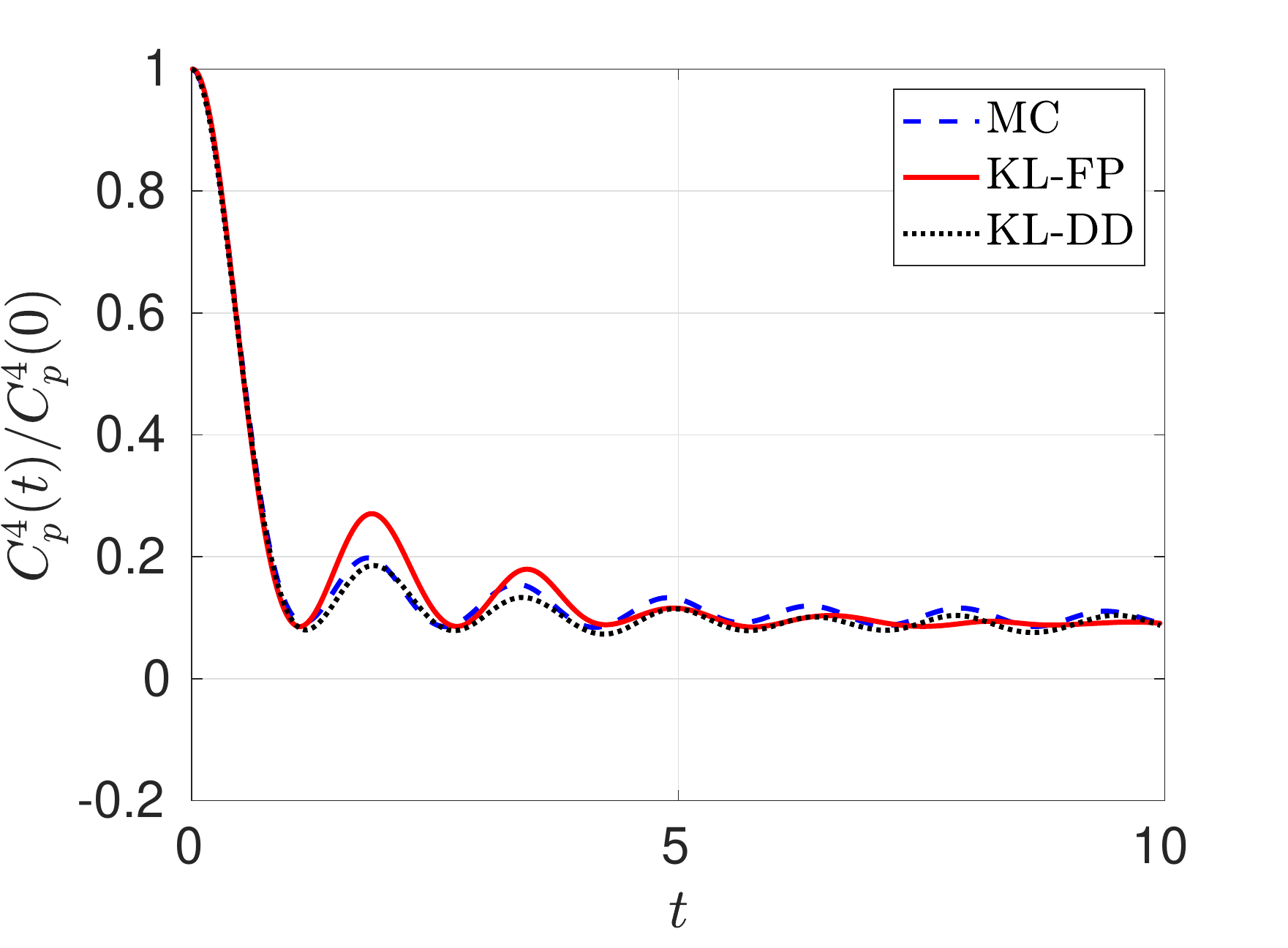}
}
\centerline{
\includegraphics[height=4.cm]{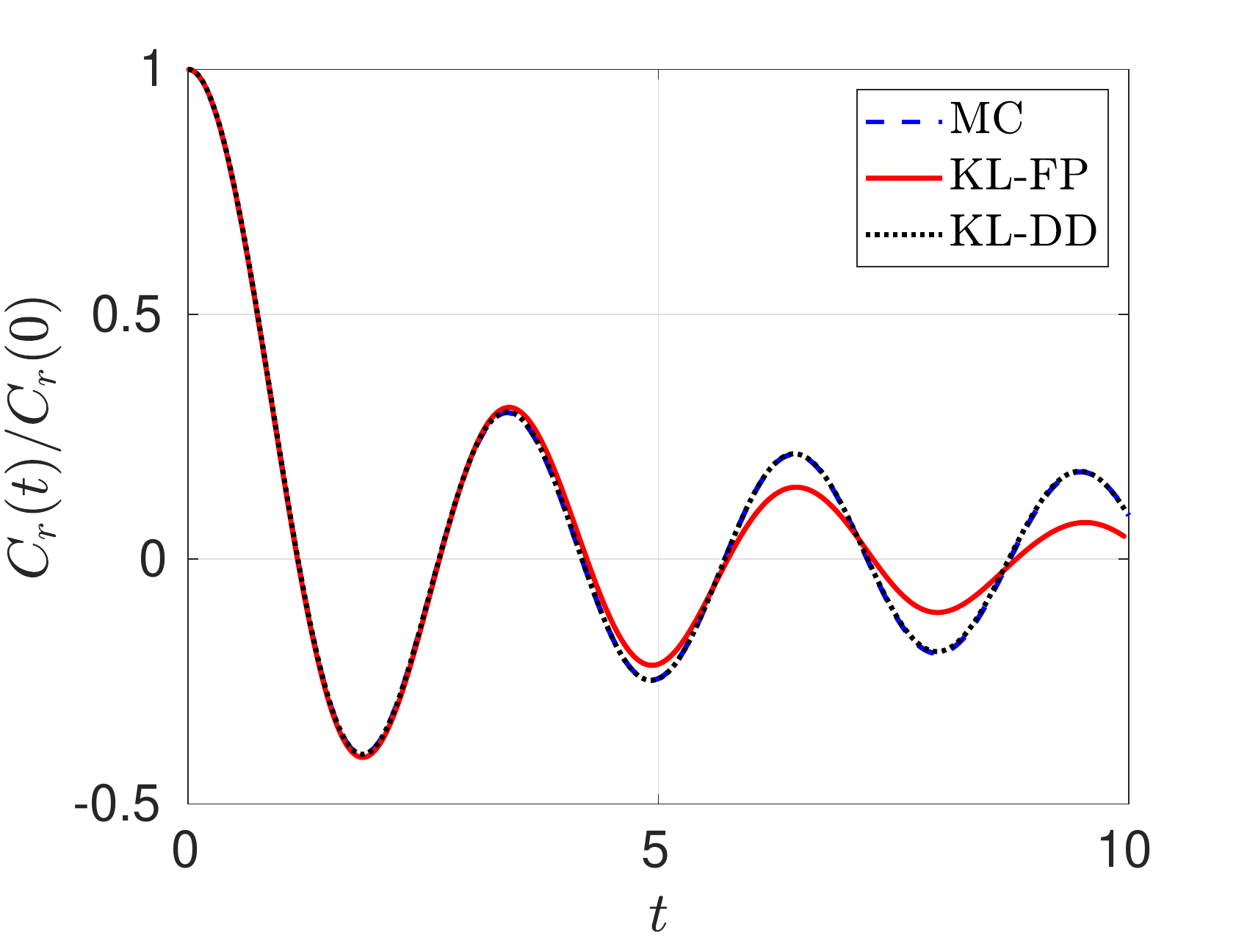} 
\includegraphics[height=4.cm]{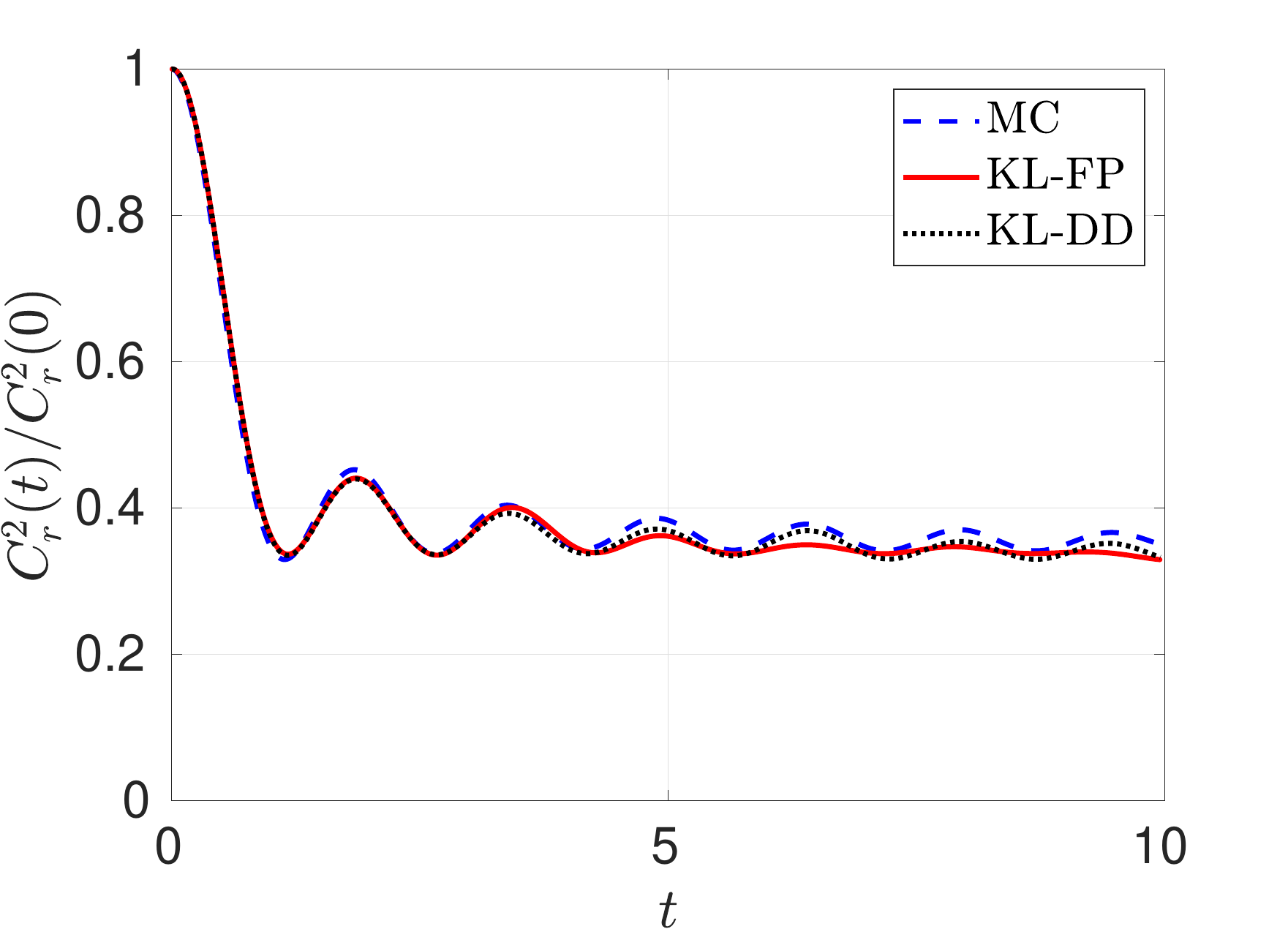}
\includegraphics[height=4.cm]{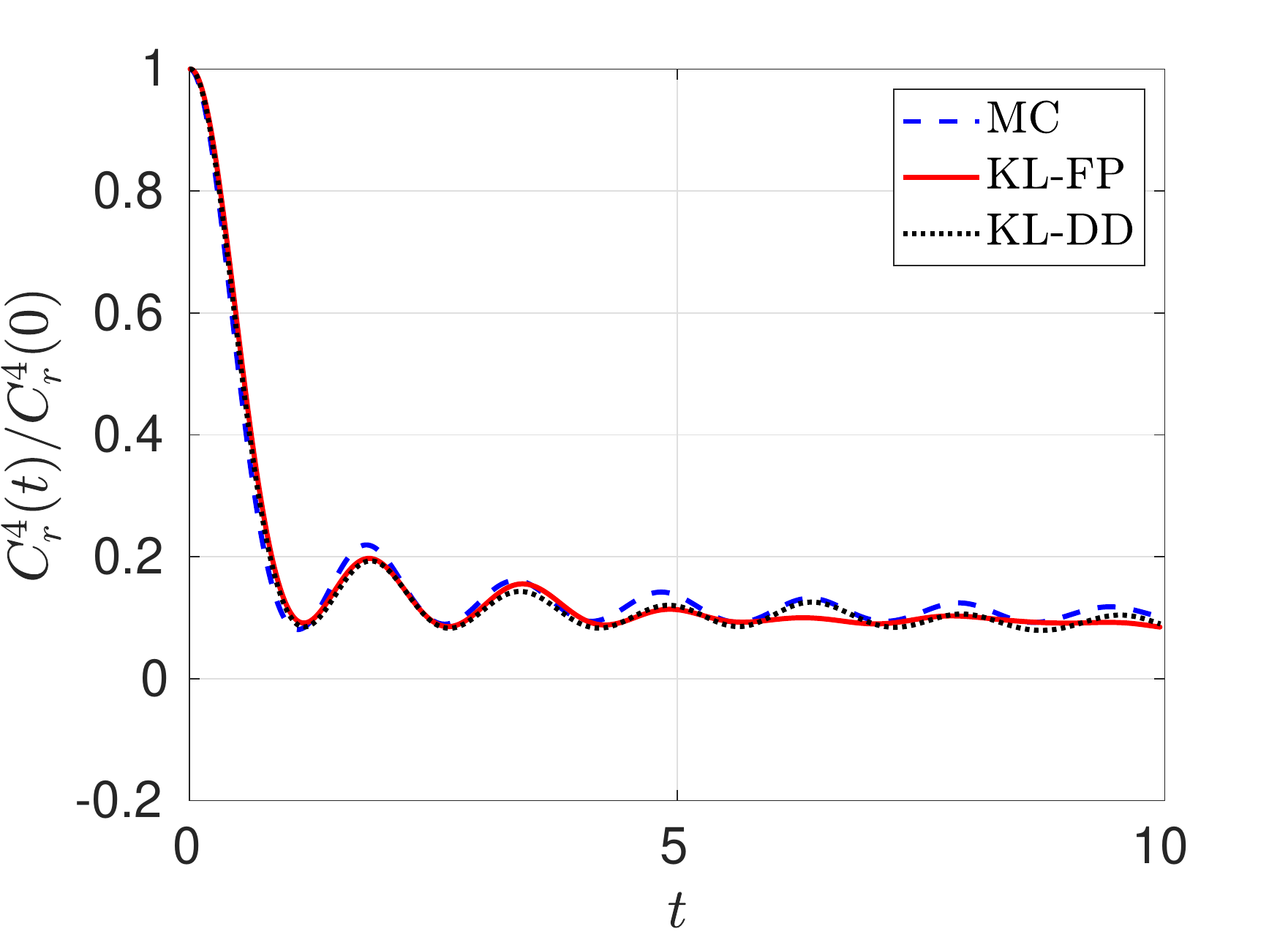}
}
\caption{Nonlinear wave equaton \eqref{eqn:nonlinear_wave1}. 
Temporal auto-correlation function of polynomial observables $p_j^{m}(t)$ (first row) $r_j^{m}(t)$ (second row) with $m=1,2,4$.
We compare results from Markov-Chain-Monte-Carlo simulation (MC), 
KL expansion based on the first-principle MZ memory kernel 
calculation \eqref{gleCC} (KL-FP), and KL expansion 
based on a data-driven estimate of the temporal auto-correlation 
function (KL-DD). 
 {The parameter $\gamma$ appearing in \eqref{Gibbs_2} 
is set to $40$, while $\alpha_1=\beta_1=1$.}}
\label{fig:C_FPU}
\end{figure}

\section{Summary}
\label{sec:conclusion}
  {
We developed a new method to approximate 
the Mori-Zwanzig (MZ) memory integral 
in generalized Langevin equations (GLEs) describing 
the evolution of smooth observables in 
high-dimensional nonlinear systems with 
local interactions. 
%We developed a new approach 
%to approximate the memory integral
%and the fluctuation term in the Mori-Zwanzig (MZ)
%equation describing the dynamics of a smooth 
%observable in a high-dimensional nonlinear 
%system with local interactions. 
 %
The new method is based on Faber operator 
series expansions \cite{zhu2018faber}, 
and a formally exact combinatorial algorithm 
that allows us to compute the expansion 
coefficients of the  MZ memory from first 
principles, i.e., based on the microscopic 
equations of motion. 
We also developed a new stochastic 
modeling technique that employs Karhunen-Lo\`eve 
expansions to represent the MZ fluctuation 
term (random noise) for systems in statistical 
equilibrium. 
We demonstrated the MZ memory calculation method 
and the MZ-KL stochastic modeling 
technique in applications 
to random wave propagation and prototype 
problems in classical statistical mechanics such as 
the Fermi-Pasta-Ulam $\beta$-model. 
We found that the proposed algorithms can 
accurately capture relaxation to statistical equilibrium in 
systems with mild nonlinearities, and in strongly 
nonlinear systems at high-temperature. 
At low temperature the Faber expansion of the 
MZ memory kernel is granted to converge only 
on a time interval that depends on the 
system and on the observable. In particular, 
Corollary 3.4.3 in \cite{zhu2018estimation} establishes
short-time convergence of the MZ-Faber memory 
approximation  for a broad class of nonlinear 
systems of the form \eqref{eqn:nonautonODE}. 
This implies that the MZ-Faber cumulant expansion 
can exhibit short-time convergence, 
meaning that it produces first-principle results that 
are accurate only for relatively short integration times. 
}

We conclude by emphasizing that the 
mathematical techniques we presented can be 
readily applied to more general systems with local 
interactions such as particle systems modeling the 
microscopic dynamics of solids and liquids
\cite{Yoshimoto2013,Li2015,Li2017}. 
This opens the possibility to build new approximation 
schemes for MZ equations and derive  
new types of coarse-grained models where 
the MZ memory is constructed 
from first-principles and the fluctuation 
term is modeled stochastically.

\vspace{0.3cm}
\noindent 
{\bf Acknowledgements} 
This research was supported by the Air Force 
Office of Scientific Research (AFOSR) 
grant FA9550-16-586-1-0092.

\appendix
 
\section{Auto-correlation function of polynomial observables}
\label{app:KLconvergence}
In this Appendix we prove that the temporal 
auto-correlation function of phase space functions 
of the form $u^n(t)=u^n(\bm x(t,\bm x_0))$, i.e., 
\begin{equation}
\langle u^{n}(0),u^{n}(t)\rangle_{\rho}\qquad n\in \mathbb{N}
\end{equation}
can be represented by replacing $u(t)$ with the KL expansion 
\eqref{KL_Sample}, and then sending $K$ to infinity. 
This result allows us to compute the auto-correlation function 
of $u^n(t)$ based on the KL expansions of $u(t)$. 

\vspace{0.2cm}
\noindent
{\bf Theorem A.1} {\em
Consider a zero-mean stationary stochastic process 
$u(t)$, $t\in[0,T]$, and assume that it has finite 
joint moments up to any desired order. Let   
\begin{equation}
u_K(t) = \sum_{k=1}^K \sqrt{\lambda_k}\xi_k e_k(t), \qquad 
\end{equation}
be the truncated Karhunen-Lo\`eve expansion of $u(t)$. 
Then 
\begin{equation}
\lim_{K\rightarrow \infty} 
\left|\langle u^{n}(0), u^{n}(t)\rangle_{\rho}-
 \langle u_{K}^{n}(0), u_{K}^{n}(t)\rangle_{\rho}\right| 
 \qquad \forall 
 n\in\mathbb{N}, 
\end{equation}
i.e., $\langle u_{K}^{n}(0), u_{K}^{n}(t)\rangle_{\rho}$ converges 
uniformly to $\langle u^{n}(0), u^{n}(t)\rangle_{\rho}$ 
as $K\rightarrow \infty$.
}

\begin{proof}
Let us define 
\begin{align}
\delta_K(t)&=\left|\langle u_{K}^{n}(t),u_{K}^{n}(0)\rangle-
\langle u^{n}(t),u^{}(0)\rangle\right|\nonumber\\
&=|
\langle u_{K}^{n}(t),u_{K}^{n}(0)\rangle
-\langle u^{n}(t),u^{n}_{K}(0)\rangle
+\langle u^{n}(t),u^{n}_{K}(0)\rangle
-\langle u^{n}(t),u^{n}(0)\rangle|\nonumber\\
&=
|\langle u_{K}^{n}(t)-u^{n}(t),u_{K}^{n}(0)\rangle
+\langle u^{n}(t),u^{n}_{K}(0)-u^{n}(0)\rangle|\nonumber\\
&\leq 
|\langle u_{K}^{n}(t)-u^{n}(t),u_{K}^{n}(0)\rangle|
+|\langle u^{n}(t),u^{n}_{K}(0)-u^{n}(0)\rangle|\label{eqn:deltaK}
\end{align} 
The first term at the right hand side is of the form
\begin{align*}
a^{n}-b^{n}=(a-b)\sum_{i=0}^{n-1}a^{i}b^{n-1-i}.
\end{align*}
By using the Cauchy-Schwarz inequality, we obtain
\begin{align*}
|\langle u_{K}^{n}(t)-u^{n}(t),u_{K}^{n}(0)\rangle|
=&\left|\langle(u_{K}(t)-u(t))\sum_{i=0}^{n-1}u_K^i(t)
u^{n-1-i}(t),u_{K}^{n}(0)\rangle\right|\\
=&\left|\langle u_{K}(t)-u(t),u_{K}^{n}(0)\sum_{i=0}^{n-1}
u_K^i(t)u^{n-1-i}(t)\rangle\right|\\
\leq&\epsilon_K(t)\left\|u_{K}^{n}(0)\sum_{i=1}^{n-1}
u_K^i(t)u^{n-1-i}(t)\right\|_{L^2},
\end{align*}
where we defined $\epsilon_K(t)=\|u_K(t)-u(t)\|_{L^2}$. It is well-known 
that $\epsilon_K(t)\rightarrow 0$ as $K\rightarrow \infty$ 
(see, e.g., \cite{le2010spectral}). 
By using the 
generalized H\"{o}lder's inequality 
$\|fg\|_{L^p}\leq\|f\|_{L^{q}}\|g\|_{L^{q}}$, where $2p=q$ and the 
Minkowski inequality, we obtain
\begin{align}\label{eqn:es1}
|\langle u_{K}^{n}(t)-u^{n}(t),u_{K}^{n}(0)\rangle|
&\leq \epsilon_K(t)\|u_K^n(0)\|_{L^4}\sum_{i=1}^{n-1}
\|u_K^i(t)u^{n-i-1}(t)\|_{L^4}\nonumber\\
&\leq \epsilon_K(t)\|u_K^n(0)\|_{L^4}\sum_{i=1}^{n-1}
\|u_K^i(t)\|_{L^8}\|u^{n-i-1}(t)\|_{L^8}=C_1\epsilon_K(t), 
\end{align}
where 
\begin{equation}
C_1 = \|u_K^n(0)\|_{L^4}
\sum_{i=1}^{n-1}\|u_K^i(t)\|_{L^8}\|u^{n-i-1}(t)\|_{L^8}
<\infty.
\end{equation}
Similarly, we have 
 {
\begin{align}\label{eqn:es2}
|\langle u^{n}(t),u_{K}^{n}(0)-u^n(0)\rangle|
\leq \epsilon_K(0)\|u^n(0)\|_{L^4}
\sum_{i=1}^{n-1}\|u_K^i(0)\|_{L^8}\|u^{n-i-1}(0)\|_{L^8}
=C_2\epsilon_K(0).
\end{align}
}
By combining \eqref{eqn:deltaK}, \eqref{eqn:es1} 
and \eqref{eqn:es2}, we obtain 
\begin{align*}
\lim_{K\rightarrow+\infty} \delta_K(t)\leq 
\lim_{K\rightarrow+\infty} C_1\epsilon_K(t)+C_2\epsilon_K(0)=0,
\end{align*}
which proves the theorem. 

\end{proof}

\newpage
%\section*{References}
%\bibliographystyle{plain}
%\bibliography{3}

\end{document}